\documentclass[smallextended,envcountsect]{svjour3} 

\smartqed 

\usepackage{graphicx}
\usepackage{a4wide}
\usepackage{enumitem}
\usepackage{color}
\usepackage{marvosym}
\usepackage{amsmath,amssymb,url}
\usepackage{algorithmicx}
\usepackage{algpseudocode}
\usepackage{tikz}
\usepackage[Algorithm,ruled]{algorithm}

\usepackage{subfigure}
\usepackage[numbers,sort]{natbib}

\journalname{JOTA}

\usepackage{comment}

\newtheorem{assumption}{Assumption}
\graphicspath{{../}} 
\newcommand{\norm}[1]{\textstyle{\Vert} #1 \textstyle{\Vert}}

\newcommand{\col}[1]{\left\{#1\right\}}

\newcommand{\co}{{\tt co}}

\newcommand{\tol}{\texttt{Tol}}
\newcommand{\inner}[2]{\langle#1,#2\rangle}

\newcommand{\E}{\mathbb{E}}
\newcommand{\R}{\Re}

\newcommand{\ind}{{\rm i}}

\newcommand{\pclarke}{\partial^{\tt C}}

\newcommand{\ff}{\mathsf{f}}

\renewcommand{\H}{\check{\mathcal{M}}}
\newcommand{\Model}{\mathcal{M}}

\newcommand{\f}{\mathfrak{f}}
\renewcommand{\c}{\mathfrak{c}}

\newcommand{\ReInf}{\mathbb{R}\cup\{+\infty\}}

\newcommand{\x}{\bar{x}}
\newcommand{\epi}[1]{\operatorname{epi}\left(#1\right)}
\newcommand{\LimInn}[1]{\operatorname{LimInn}\left(#1\right)}
\newcommand{\LimOut}[1]{\operatorname{LimOut}\left(#1\right)}

\renewcommand{\Re}{\mathbb{R}}

\newcommand{\mybox}{{\hfill $\square$}}

\newcommand{\wlo}[1]{#1}

\definecolor{gris}{gray}{0.6}
\newcommand{\mytriangle}{{\color{gris}\hspace*{\fill}\scriptsize\ensuremath{\blacktriangleleft}}}

\newtheorem{prop}{Proposition}

\begin{document}
\title{A Proximal-Type Method for Nonsmooth and Nonconvex Constrained Minimization Problems}
\titlerunning{A Proximal-Type Method}

\author{Gregorio M. Sempere$^1$, Welington de Oliveira$^1$, Johannes O. Royset$^2$}
\authorrunning{G.M. Sempere et al.}

\institute{$1$ MINES ParisTech, PSL -- Research University, CMA -- Centre de Math\'ematiques Appliqu\'ees,  France \\
 \email{\url{gregorio.martinez_sempere@minesparis.psl.eu}}\\
$2$ University of Southern California
}

\date{\today}

\maketitle

\begin{abstract}
This work proposes an implementable proximal-type method for a broad class of optimization problems involving nonsmooth and nonconvex objective and constraint functions. In contrast to existing methods that rely on an ad hoc model approximating the nonconvex functions, our approach can work with a nonconvex model constructed by the pointwise minimum of finitely many convex models. The latter can be chosen with reasonable flexibility to better fit the underlying functions' structure. We provide a unifying framework and analysis covering several subclasses of composite optimization problems and show that our method computes points satisfying certain necessary optimality conditions, which we will call model criticality.  Depending on the specific model being used, our general concept of criticality boils down to standard necessary optimality conditions. Numerical experiments on some stochastic reliability-based optimization problems illustrate the practical performance of the method.
\end{abstract}


\section{Introduction}\label{sec:introduction}

Modern optimization problems, such as those in stochastic optimization, machine learning, uncertainty quantification, and others, often involve functions that fail to be smooth and convex. Nevertheless, many such problems possess some structure that can be exploited in numerical optimization. Existing nonsmooth and nonconvex optimization methods typically rely on an ad hoc model exploiting a specific (composite) structure \cite{Burke_1995,Sagastizabal_2013,Lewis_2016,Cui_Pang_Sen_2018,Liu_Cui_2020,Cui_Pang_Liu_2022,AragonPerezTorregrosa_2023}. Most such methods deal with the convex-constrained setting.
For instance, \cite{Burke_1985,Burke_1995,Sagastizabal_2012} exploits convexity and smoothness of certain components functions yielding the composite one, \cite{Lewis_2016} and \cite{AragonPerezTorregrosa_2023} profit from a certain degree of smoothness, \cite{Pham_2005,Oliveira_2019,Cui_Pang_Sen_2018} from a Difference-of-Convex structure, and \cite{Ackooij_Oliveira_2020} exploits the difference-of-locally-Lipschitz format. 
A framework for constructing and analyzing approximations of optimization problems with a composite structure is given in \cite{Royset_2023}, and \cite[Chapter 7]{Cui_Pang_2022} presents several algorithms based on the principle of surrogation.

This work considers a broader class of composite optimization problems involving nonsmooth and nonconvex objective and constraint functions. The investigated setting exploits the available structure so that the problem can be locally approximated by a model even if subgradients are unavailable. This is the case for some composite functions whose first-order information is available for the components but not necessarily for the composition. Applications include Difference-of-Convex (DC) problems, classification, and sample average approximations of chance-constrained and buffered probability optimization problems.

Some optimization algorithms exist for handling different families of problems fitting into our setting.
For instance, \cite{Apkarian2008} deals with the so-called mixed $H_2/H_\infty$ control problem, \cite{Cui_Pang_Liu_2022} focuses on the particular case of chance-constrained problems, and the DC-constrained case is investigated in \cite{Pham_Penalization_2012,Strekalovsky_Minarchenko_2017,Pang_2017,Montonen_Joki_2018,Cui_2018,Javal_2021}. A broader setting than DC was recently investigated in \cite{Ksenia_2023}, where the objective and constraint functions can be expressed as the difference of convex and weakly-convex functions.
 These papers build different models to provide implementable algorithms in their respective settings. The main idea is to exploit the functions' structure to approximate the problem locally. For instance, complicating (component) functions are linearized in \cite{Pham_Penalization_2012,Strekalovsky_Minarchenko_2017,Pang_2017,Montonen_Joki_2018,Cui_2018,Javal_2021,
 Ksenia_2023} whereas higher-order approximations are considered in \cite{AragonPerezMordukhovich_2023} under structured nonsmoothness. In addition, these models must accommodate nonconvex constraints. Standard ways to face constrained optimization problems are through penalization \cite{Pham_Penalization_2012,Byun_Oliveira_Royset_2023} and \cite[\S  9]{Cui_Pang_2022}, linearizations of constraints \cite{Pang_2017,Cui_2018,Ackooij_Oliveira_OMS_2019}, and, as considered in this work,  the so-called \textit{improvement function} \cite{Montonen_Joki_2018,Javal_2021,Sagastizabal_Solodov_2005,Apkarian2008} that jointly combines objective and constraint functions. 
 
 Each of these references develops proximal-type algorithms to compute a point satisfying some necessary optimality condition denoted by criticality. 
 Many definitions of criticality exist in nonsmooth and nonconvex optimization, which are closely related to the working model for the problem.
 Therefore, each of these references exploits a model and provides a dedicated algorithm and convergence analysis. Differently, this work proposes a general yet implementable proximal-type algorithm with guaranteed convergence to critical points. Our study is comprehensive, and the type of criticality we get depends explicitly on the model's properties. The latter can be chosen with reasonable freedom as long as some conditions hold to ensure its tractability. 
 
In a nutshell, our approach is based on a nonconvex model (for the improvement function) constructed by the pointwise minimum of finitely many \emph{convex models}, which are obtained using the available information about the functions. For instance,  if some of the available oracles for the component functions can provide us with more than one subgradient for every given point, we might benefit from this fact to construct more than one convex model approximating the functions locally. 
The pointwise minimum of such convex models leads to nonconvex but computationally solvable subproblems yielding new iterates.
 Thanks to this more general model, our method and analysis cover several classes of (nonlinearly constrained) composite optimization problems.
Our analysis is based on the subdifferentiability (see \cite[\S 2]{Clarke_1987}, \cite[\S  4]{RoysetWets_2022} and \cite[\S  4]{Cui_Pang_2022}) and  epigraphical nesting (see \cite[\S 6.C]{RoysetWets_2022} and \cite{HigleSen_1992})  theories. The latter is a known relaxation of epi-convergence (see \cite[\S  5]{RoysetWets_2022} and \cite[\S  7]{Rockafellar_Wets_2009}) with a broader field of use that allows us to accommodate the convergence analysis of implementable algorithms, such the one proposed in Section \ref{sec:algo} below.

The manuscript is organized as follows. Section \ref{sec:pre}  recalls standard definitions and fundamental concepts for our development, such as the improvement function and epigraphical nesting. Section \ref{sec:crit} presents the basic assumptions of the working model, together with the type of criticality to attain. Our algorithm is given in Section \ref{sec:algo} and its analysis relies on epigraphical nesting. Sections \ref{sec:appl}, \ref{sec:dcset}, and \ref{sec:comp} present some examples of classes of functions for which we can provide implementable models for our algorithm.
Section \ref{sec:num} illustrates the numerical performance of our approach with two stochastic optimization problems. Finally, Section \ref{sec:conc} closes the manuscript with some concluding remarks.

\paragraph{Notation:}
We consider $f:\R^n\to \R$ as a real-valued function, while we use $\ff:\R^n\to\ReInf$ to refer to an extended real-valued function with a nonempty domain. We denote by $X$ a convex and closed set and by ${\rm i}_X$ the \textit{indicator function} of $X$, that equals $0$ if $x\in X$ and equals $+\infty$ otherwise. Also, we denote by ${\pmb 1}_A$ the \textit{characteristic function} of $A$, that equals $1$ if $x\in A$ and equals $0$ otherwise. The \textit{normal cone of the \emph{(closed and convex)} set $X$ at $\x\in X$} is
$$N_X(\x)=\{\xi\in\R^n:\inner{\xi}{x-\x}\leq 0,\;\forall x\in X\},$$
and $N_X(\x)=\emptyset$ if $\x\notin X$.
Given an optimization problem such as $\min_{x\in X}f(x)$, where $f$ is locally Lipschitz, we say that $\x\in X$ is a \textit{stationary point} if $0\in \pclarke f(\x)+N_X(\x)$, where 
$$
\pclarke f(\x)=\co\{\xi\in\R^n:\exists\; x^k\to\x, \nabla f(x^k)\to\xi\}
$$
is the \textit{Clarke subdifferential of $f$ at $\x$}. If $f$ happens to be convex, it is well-known that $\pclarke f(\x)=\partial f(\x)$, where
$$
\partial f(\x)=\{\xi\in\R^n:\inner{\xi}{x-\x}\leq f(x)-f(\x), \forall x\in\R^n\}.
$$
We consider the \textit{tangent cone to a closed set $A$ at a point $\x\in A$} in the Bouligand sense:
$$\mathcal{T}_{A}(\x)=\left\{d\in\R^n:\exists \{y^k\}_k\subseteq A\;,t_k>0\;, y^k\to \x\;,\;t_k\to0:d=\lim_{k\to\infty}\frac{y^k-\x}{t_k}\right\}.$$
We denote the \textit{epigraph} of a function $\ff:\R^n\to\ReInf$ by $\epi{\ff}=\{(x,\alpha)\in\R^{n+1}:\ff(x)\leq \alpha\}$, and the \textit{inner} and \textit{outer limits} of sequence of sets $D_k\subseteq\R^n$ by $\LimInn{D_k}=\{\x\in\R^n:\exists y^k\in D_k, y^k\to\x\}$ and $\LimOut{D_k}=\{\x\in\R^n:\exists \mathcal{K}\subset\mathbb N,\; y^k\in D_k, \lim_{k\in\mathcal{K}\to\infty}y^k=\x\}$.

\section{Problem Statement and Prerequisites}\label{sec:pre}
Given locally Lipschitz continuous functions $f,c:\R^n\to\R$  and $\emptyset\neq X\subseteq\R^n$, a convex and closed set, we are interested in optimization problems such as
\begin{equation}\label{pb}
\min_{x\in X} \; f(x) \quad
    \text{s.t.}\quad c(x)\leq 0.
\end{equation}
We assume that the domains of both $f$ and $c$ are the whole space $\R^n$, however, this could be relaxed to the set $X$ being contained in the interior of both domains. Functions $f$ and $c$ need not be convex or differentiable, which makes this problem challenging from the numerical optimization point of view unless some favorable structure is assumed. Even when the problem is well-structured, e.g., DC problems, computing a solution with a certificate of global optimality is out of reach in real-life applications such as those coming from stochastic optimization or machine learning. This is why much research focuses on local-solution methods, i.e., algorithms capable of computing a point satisfying certain necessary optimality conditions, denoted by criticality.
What we mean by \textit{criticality} is specified in Section \ref{sec:crit} below.

\paragraph{\textbf{Improvement function}.}
We now recall an important mathematical tool in nonlinearly-constrained optimization, the so-called  \textit{improvement function} \cite{Sagastizabal_Solodov_2005}:
\begin{equation}\label{eq:IF}
H(y;x)=\max\{f(y)-\tau_f(x)\;,\;c(y)\}\quad \mbox{with}\quad
\tau_f(x)=f(x)+\rho [c(x)]_+
\end{equation}
and given parameter $\rho\geq 0$. 
 We highlight that dealing with a difficult constraint via an improvement function is a well-known and successful strategy in the nonsmooth optimization literature \cite{Sagastizabal_Solodov_2005,Apkarian2008,Montonen_Joki_2018,Javal_2021,KseniaPaul_2023}. 
 It is an important alternative to penalty approaches, which generally require laborious work of (penalty) parameter tuning.
 In contrast, the parameter $\rho$ is simple to tune as it
 can be chosen such that the magnitudes of both functions $f$ and $c$ are close\footnote{ For instance, a possible choice for $\rho$ is $\rho=|f(x^0)|/(1+|c(x^0)|)$, with $x^0 \in X$ being a given initial point to the optimization algorithm.}.
The following properties of $H$, which are trivial extensions of \cite[Lem. 5.1]{Apkarian2008} to our setting,  are handy. 
\begin{lemma}[Lem. 5.1 in \cite{Apkarian2008} and Lem. 1 in \cite{Javal_2021}.]\label{lemma:ineq}  The following hold.
\begin{enumerate}[label=(\roman*), start=1]
    \item Let $\mu\geq 0$ and $x,y$ be two points in $X$ satisfying $H(y;x)+\frac{\mu}{2}\norm{y-x}^2\leq H(x;x)$:   
    \begin{enumerate}[label={(i.\arabic*)}]
    \item If $c(x)\leq 0$, then $c(y)+\frac{\mu}{2}\norm{y-x}^2\leq 0$ and $f(y)+\frac{\mu}{2}\norm{y-x}^2\leq f(x)$.
    \item If $c(x)>0$, then $c(y)+\frac{\mu}{2}\norm{y-x}^2\leq c(x)$.
    \end{enumerate}
 \item If $\x \in X$ is a local solution to \eqref{pb}, then $\x$ is also a local
 minimizer of $H(x;\x)$ over $X$. 
  \end{enumerate}  \mytriangle 
\end{lemma}

\paragraph{\textbf{Epi-convergence and epigraphical nesting}.}
We now recall the concept of \emph{epigraphical nesting}, which is a relaxation of \emph{epi-convergence}  \cite[\S 7]{Rockafellar_Wets_2009}. Epigraphical nesting was proposed in \cite{HigleSen_1992} as a tool for analyzing minimization algorithms (see also \cite[\S 6.C]{RoysetWets_2022}).

\begin{definition}[Def. 7.1 in {\cite{Rockafellar_Wets_2009}} and Thm. 4.15 in {\cite{RoysetWets_2022}}]\label{def:epi}
    Given $\ff^k,\ff:\R^n\to\ReInf$, the sequence $\{\ff^k\}_k$ is said to epi-converge to $\ff$, denoted by $\ff^k\xrightarrow{e}\ff$,  if the following condition holds:
$$
\LimOut{\epi{\ff^k}}\subseteq \epi{\ff} \subseteq \LimInn{\epi{\ff^k}}.
$$
    In that case, we refer to $\ff$ as the \textit{epi-limit} of the sequence $\{\ff^k\}_k$.\mytriangle
\end{definition}
Note that, by definition of inner and outer limits, the above inclusions are equivalent to all three sets to be equal. 
We choose the above characterization of epi-convergence as it gives a clear view of how it relates to epigraphical nesting. However, expressing these inclusions in terms of convergence of $\{\ff^k\}_k$ is handy for the following proofs.
\begin{lemma}\label{lemma:epi2}
    Given $f^k,f:\R^n\to\R$ and $X\subseteq\R^n$ a closed set, it holds:
    \begin{enumerate}[label=(\roman*)]
        \item $\LimOut{\epi{f^k+{\rm i}_X}}\subseteq \epi{f+{\rm i}_X}$ if, and only if, for all $x\in X$ and for each sequence $\{x^k\}_k\subseteq X$ such that $x^k\to x$,
        $$
        f(x)\leq\liminf_k f^k(x^k) .
        $$
        \item $\epi{f+{\rm i}_X}\subseteq \LimInn{\epi{f^k+{\rm i}_X}}$ if, and only if, for all $x\in X$ there exists a sequence $\{y^k\}_k\subseteq X$ such that $y^k\to x$ and
        $$
        \limsup_kf^k(y^k)\leq f(x).
        $$
    \end{enumerate}
\end{lemma}
{\bf Proof}
Let us recall {\cite[Prop. 7.2]{Rockafellar_Wets_2009}}, which states essentially the same result for extended-valued functions. In that setting, there is no need to  focus on points and subsequences in $X$. When applying that result to our case, it turns out that we do not need to investigate $x\notin X$, as both inequalities
        $$
        \limsup_kf^k(x^k)+{\rm i}_X(x^k)\leq f(x)+{\rm i}_X(x)\leq\liminf_k f^k(x^k)+{\rm i}_X(x^k) 
        $$
would  hold for any sequence $x^k\to x$, as ${\rm i}_X(x^k)={\rm i}_X(x)=+\infty$ eventually. Let us now assume that $x\in X$. To prove item (i), note that $\LimOut{\epi{f^k+{\rm i}_X}}\subseteq\epi{f+{\rm i}_X}$ is equivalent to $\liminf_kf^k(x^k)+{\rm i}_X(x^k)\geq f(x)+{\rm i}_X(x)$ for all $x^k\to x$ (due to  {\cite[Prop. 7.2]{Rockafellar_Wets_2009}}), if there were a subsequence that happened to be out of $X$, the inequality would hold trivially. Thus, it is enough to consider only for those sequences $\{x^k\}_k\subseteq X$. 

\wlo{The proof of item (ii) is analogous, mutatis mutandis.}
    \mybox

Note that for the sequence $\{y^k\}_k$ above, we have $\lim_k f^k(y^k)=f(x)$.
Useful properties can be proven for a sequence of functions that is epi-convergent (see \cite[\S  7]{Rockafellar_Wets_2009}). For instance, all cluster points of the sequence of approximate minimizers of $f^k+{\rm i}_X$ minimize $f+{\rm i}_X$ under mild assumptions \cite[Thm. 7.31 (b)]{Rockafellar_Wets_1986}. 
Although the importance of this result in optimization, algorithmically designing an epi-convergent sequence of functions is not always a simple task. 
For instance, it is well known that even basic convergent algorithms, such as the cutting-plane method in convex optimization, fail to construct a sequence of functions that epi-converges to the objective function.
The main obstacle in the concept of epi-convergence is that it requires a sequence of functions for which we have both over- and under-estimator subsequences around every point. It was shown in \cite{HigleSen_1992} that, for algorithmic design purposes, it is enough to have only the right-most inclusion in Definition~\ref{def:epi}, i.e.,
$$
\epi{\ff}\subseteq \LimInn{\epi{\ff^k}}.
$$
 This condition is known as \emph{epigraphical nesting}, and it ensures that the sequence of functions estimates $\ff$ from below. This concept is clearly weaker than epi-convergence, but it still proves useful in optimization.
 \wlo{
\begin{prop}[{\cite[Prop. 6.14]{RoysetWets_2022}}]\label{prop:conv}
    For $\ff^k,\ff:\R^n\to\ReInf$ and a sequence $\{x^k\}_k$, suppose that 
    $$
    \epi{\ff}\subseteq\LimInn{\epi{\ff^k}}.
    $$
    If $\epsilon^k\to0$, for some $K\subseteq\mathbb{N}$ such that  $x^k\in\epsilon^k$-$\arg\min \ff^k\to_{K} \x$ and $\ff^k(x^k)\to_K\ff(\x)$, then one also has $\x\in\arg\min\ff$.
\end{prop}}

\section{Model and Necessary Optimality Conditions}\label{sec:crit}
In nonconvex nonsmooth optimization,
many definitions of stationary points exist. A well-known definition is Bouligand stationarity \cite{Pang_2007}:  
a point $\x \in X^c:=\{x\in X:\; c(x)\leq 0\}$ is $B$-stationary to problem~\eqref{pb} if the directional derivative of $f$ with respect to $d$ is non-negative
for all $d \in \mathcal{T}_{X^c}(\x)$. 
 Obtaining an explicit expression of $\mathcal{T}_{X^c}(\x )$ is not simple, let alone an implementable one. We refer the interested reader to \cite{Pang_2017} and \cite{Ackooij_Oliveira_2022} for implementable expressions of $\mathcal{T}_{X^c}(\x )$ in particular settings of DC programming, see also Section \ref{sec:dcset} below.
For this reason, different necessary optimality conditions must come into play.

Given item {(ii)} of Lemma~\ref{lemma:ineq}, we could then use  $\x \in \arg\min_{x\in X}H(x;\x)$ as a necessary optimality condition for $\x \in X$ to solve problem~\eqref{pb}. However, for the class of problems we have in mind, working directly with the improvement function is still challenging. The reason is that  
it may happen in practice that (first-order) oracles for $f$ and $c$ are not available, but for functions composing them; see Sections \ref{sec:appl}-\ref{sec:comp} below. 
Our main assumption is that we can
construct, with the available oracle information, a model $\Model:\R^n\times\R^n\to\R$ for $H$ satisfying the following hypothesis.
\begin{assumption}[Model's requirements.]\label{assump-model}
\mbox{}
For all $x \in X$ fixed, let $A(x)$ be a finite index set and $\Model_a:\R^n\times\R^n\to \R$ be a real-valued function such that  $\Model_a(\cdot;x)$ is convex for all $a\in A(x)$.
We take $\Model(\cdot;x):=\min_{a \in A(x)}\; \Model_a(\cdot;x)$ as a model for $H(\cdot;x)$ in \eqref{eq:IF}
provided that the following conditions hold:
\begin{enumerate}[label=(\alph*)]
\item $\Model(x;x)=H(x;x)$ for all $x \in X$;
\item there exists a constant $\bar \mu \in [0,+\infty)$  such that $
H(y;x)\leq  \Model(y;x)+\frac{\bar \mu}{2}\norm{y-x}^{2}$ for all $y,x \in X$.
\end{enumerate}
\mytriangle
\end{assumption}
The hypotheses on $A(x)$ and $\Model_a$ are natural and helpful when $H$ can be written as the pointwise minimum of \wlo{a finite set of ``simple'' functions} or when oracles assessing component functions in $H$ can provide more than one subgradient. If none are the case, we can take $A(x)$ being a singleton so $\Model(\cdot;x)$ is a convex model. The case  $|A(x)|>1$ gives rise to a nonconvex (DC) model.
Assumption (a) is also a natural one and (b) requests the model, plus a quadratic term, to overestimate the improvement function. 
Although this last condition looks too restrictive, the following sections show otherwise: the modulus $\bar \mu>0$ need not be known, and we can build proper models for which this assumption is satisfied.
\wlo{Note that whether $\Model=H$, then Assumption \textit{(b)} reduces to \textit{weak convexity}, a well-known concept for which several models have been already proposed (see for instance \cite{atenas_2021,druvyatskiy_2021}). Moreover, for $\mu\geq \bar\mu$, the function $\Model(y;x)+\frac{\mu}{2}\|y-x\|^2$ defines a surrogation function in the sense of \cite[Def. 7.1.1]{Cui_Pang_2022}}.

\begin{definition}[Model criticality.]\label{def:crit}
Let $\Model:\R^n\times\R^n\to\R$ be a model satisfying Assumption~\ref{assump-model}. 
A point $\x \in X$ is said to be \emph{model critical} (M-critical for short) to problem~\eqref{pb} if  
\[ \x\in\displaystyle\arg\min_{x\in  X}\Model_a(x;\x)
\quad \forall\; a\in A_H(\x):=\col{\alpha\in A(\x): \, \Model_\alpha(\x;\x)=H(\x;\x)}.
\]
Furthermore, $\x $ is said to be \emph{feasible-model critical} (FM-critical)
if it is M-critical and  $c(\x) \leq 0$.
    \mytriangle
\end{definition}

\noindent This new concept of criticality is a necessary optimality condition for $\x \in X$ to solve problem~\eqref{pb}.

\begin{theorem}\label{thm:model}
Let $H$ be the improvement function given in~\eqref{eq:IF} and $\Model$ be a model satisfying Assumption~\ref{assump-model}. If $\x$ is a local solution to problem~\eqref{pb}, then $\x$ is FM-critical.
\end{theorem}
\begin{proof}
If $\x$ is a local solution to problem~\eqref{pb}, then Lemma~\ref{lemma:ineq} (ii) asserts that \wlo{there exists a convex neighbourhood $\mathcal{N}$ of $\x$ such that} $\x\in\arg\min_{x\in X\wlo{\cap\mathcal{N}}} H(x;\x)$.
Let $\bar \mu> 0$ be as in Assumption~\ref{assump-model}(b). Then,
\begin{align*}
 H(\x ;\x )=\Model(\x ;\x ) &\geq \displaystyle \min_{x\in X\wlo{\cap\mathcal{N}}}\; \Model(x;\x )+\frac{\bar \mu}{2}\norm{x-\x }^{2}
   \geq \displaystyle \min_{x \in X\wlo{\cap\mathcal{N}}}\; H(x;\x) = H(\x ;\x ).
\end{align*}
Thus,
\begin{align*}
H(\x ;\x )&=\displaystyle \min_{x\in X\wlo{\cap\mathcal{N}}}\; \Model(x;\x )+\frac{\bar \mu}{2}\norm{x-\x }^{2}\\
& = \displaystyle \min_{x\in X\wlo{\cap\mathcal{N}}}\;\min_{a\in A(\x)}\; \Model_a(x;\x )+\frac{\bar \mu}{2}\norm{x-\x }^{2}\\
&\leq \displaystyle \min_{x\in X\wlo{\cap\mathcal{N}}}\;\min_{a\in A_H(\x)}\; \Model_a(x;\x )+\frac{\bar \mu}{2}\norm{x-\x }^{2}\\
&= \displaystyle \min_{a\in A_H(\x)}\;\min_{x\in X\wlo{\cap\mathcal{N}}}\; \Model_a(x;\x )+\frac{\bar \mu}{2}\norm{x-\x }^{2}\\
&\leq\wlo{\min_{x\in X\cap\wlo{\mathcal{N}}}\; \Model_a(x;\x )+\frac{\bar \mu}{2}\norm{x-\x }^{2},\quad \forall a \in A_H(\x)}\\
& \leq  \Model_a(\x;\x) =H(\x ;\x ),
\end{align*}
showing that $\x \in \arg\min_{x\in X\wlo{\cap\mathcal{N}}}\; \Model_a(x;\x )+\frac{\bar \mu}{2}\norm{x-\x }^{2}$ for all $a \in A_H(\x)$\wlo{, i.e., $\x$ is a fixed point of the proximal operator of $\frac{1}{\bar \mu}\Model_a(\cdot;\x)+{\rm i}_{X\cap\mathcal{N}}$; which is convex by assumption. Thus, $\x$ is a minimizer of $\Model_a(\cdot;\x)+{\rm i}_{X\cap\mathcal{N}}$.
Convexity of  $\Model_a(\cdot;\x)$ also implies that $\x$ not only minimizes $\Model_a(\cdot;\x)$ over $X\cap\mathcal{N}$, which is a neighbourhood of $\x$ in $X$, but over the whole set $X$. So,
$\x \in\arg\min_{x\in X}\; \Model_a(x;\x )$
for all $a \in A_H(\x)$, i.e., $\x$ is FM-critical.}
\mybox
\end{proof}

\section{Algorithm and Convergence Analysis}\label{sec:algo}

This section provides a proximal-like algorithm to compute  M-critical points under certain conditions. 
The approach's working horse is a model $\Model:\R^n\times \R^n\to\R$ satisfying Assumption~\ref{assump-model}. We leave open the specific choice for the convex models $\Model_a$ composing $\Model$, as we intend to make Algorithm~\ref{algo} below applicable to a broad class of optimization problems of the form~\eqref{pb}.
Under Assumption~\ref{assump-model},  we show that
subproblem~\eqref{spbm} below can be easily solved up to a certain accuracy $\epsilon^k\geq 0$ at every iteration $k$. The epigraphical nesting concept recalled in Section~\ref{sec:pre} then will be handy to show that Algorithm~\ref{algo} asymptotically computes an M-critical point. If the algorithm starts with a feasible point, then it is ensured to compute an FM-critical point.
\begin{algorithm}
\caption{Proximal-like Method with Improvement Function to Compute Critical Points}
\label{algo}
\begin{algorithmic}[1]
{\footnotesize
\State  Let $x^0 \in X$,  $\kappa \in (0,1)$, $\mu^0\geq \kappa$, and $\lambda\in [0,\kappa)$ be given
\For{$k = 0, 1, 2, \dots$}
\State Let  $y^k$ be an $\epsilon^k$-solution to   \label{lin:yk}
    \begin{equation}\label{spbm}
    \min_{y \in X}\;{\Model(y;x^k)  + \frac{\mu^k}{2}\norm{y-x^k}^2,}\quad\qquad \mbox{with}\qquad\quad 0\leq\epsilon^k\leq\frac{\lambda}{2}\norm{y^k-x^k}^2
    \end{equation}
\If{$\norm{ y^k - x^k}\leq \tol$}
\State Stop: and return $x^k$ 
\EndIf
\If{$H(y^k;x^k)\leq H(x^k;x^k)-  \frac{\kappa-\lambda}{2}{\norm{y^k-x^k}^2}$}
\State Serious step: set $x^{k+1} = y^k$ and $\mu^{k+1}=\mu^k$ \label{SS} 
\Else
\State Null step: set $ x^{k+1}=x^k$ and $\mu^{k+1}>\mu^k$ \label{NS}
\EndIf
\EndFor
}
  \end{algorithmic}
\end{algorithm}

In Algorithm~\ref{algo}, the model 
is updated whenever a new serious step is performed: in this case, the so-called stability center $x^k$ is changed in $\Model(\cdot;x^k)$. For this task, the algorithm will need oracles providing (partial) information on functions $f$ and $c$ (see equation \eqref{mod:sum_max} below for an example of how $x^k$ may impact $\Model$).

 Only $\mu^k$ changes after a null step. Assumption~\ref{assump-model}(b) is crucial to ensure that once $\mu^k$ reaches a certain threshold $\bar \mu>0$, only serious steps are performed. This is why we do not allow $\mu^k$ to decrease.

Line~\ref{lin:yk} requires solving inexactly a nonsmooth optimization problem, which we assume has a solution (e.g. if $X$ is compact). The precision $\epsilon^k$ depends on the iterate $y^k$ and, thus, the optimization solver employed to solve~\eqref{spbm} does not receive $\epsilon^k$ as an input. However, this will not be an issue in practice, as we explain in the following remark.

\begin{remark} \label{remark:subproblems}
\wlo{As in \cite[Algo. 1]{Pang_2017},} at every iteration $k$ subproblem~\eqref{spbm} can be equivalently decomposed in a finite family of convex optimization problems that can be solved by bundle methods \cite{Correa_Lemarechal_1993}. As a result, we can globally solve the nonconvex subproblem~\eqref{spbm} by solving the following $|A(x^k)|$ convex subproblems
\begin{equation}\label{eq:subproblems}
y_{a}^k:=\displaystyle \arg\min_{x\in X} \Model_{a}(x;x^k) +\frac{\mu^k}{2}\norm{x-x^k}^2,
\end{equation}
and letting $y^k:=y^k_{a^*}$, with $a^*$ the index yielding the 
minimum functional value above.  In Appendix~\ref{ap:bundle} we rely on \cite{Correa_Lemarechal_1993} to detail a bundle-like algorithm to solve the above subproblems up to tolerance $\epsilon^k\geq 0$ satisfying the condition on the right of equation~\eqref{spbm}.
\end{remark}

\begin{theorem}\label{theo:conv}
Consider problem~\eqref{pb}, a model $\Model$ satisfying Assumption~\ref{assump-model}, and suppose that subproblem~\eqref{spbm} is solvable for all $k$.
 Let $\{x^k\}_k$
be the sequence produced by Algorithm~\ref{algo} with $\tol=0$. Furthermore, suppose  that
\begin{enumerate}[itemsep=2pt, topsep=2pt, label=(\roman*), leftmargin=1.0cm, start=1]
\item $\{x^k\}_k$ is bounded;
\item after a null step, the rule to update the prox-parameter ensures that $\mu^{k+1}>\mu^k +\delta$ for some constant $\delta>0$;
\item for any cluster  point $\x$ of $\{x^k\}_k$, there exists a subsequence  $\{x^{k_i}\}_i$  such that $\lim_{i} x^{k_i} =\x$  and
 $\epi{\bar\Model(\cdot;\x)+{\rm i}_X}\subseteq\LimInn{\epi{\Model(\cdot;x^{k_i})+{\rm i}_X}}$, with $\bar\Model(\cdot;\x):=\min_{a \in A_H(\x)}\Model_a(\cdot;\x)$.
\end{enumerate} 
Then $\x$ is M-critical, which upgrades to FM-critical if $\x$ is feasible (that is certainly true if the algorithm starts with a feasible point).
\end{theorem}
\begin{proof}
First,  observe that if the algorithm stops at iteration $k$ with $\tol=0$, then $y^k=x^k$ and, thus, condition~\eqref{spbm} 
implies that $\x=x^k$ is an exact solution to subproblem~\eqref{spbm}.
Thus,
\begin{align*}
H(\x;\x)&=\Model(\x;\x)=\min_{x\in X}\; \Model(x;\x)+\frac{\mu^k}{2}\norm{x-\x}^2\\
&\leq \min_{x\in X}\; \min_{a\in A_H(\x)}\; \Model_a(x;\x)+\frac{\mu^k}{2}\norm{x-\x}^2\\
&\leq \min_{x\in X}\;  \Model_a(x;\x)+\frac{\mu^k}{2}\norm{x-\x}^2\quad \forall\; a\in A_H(\x)\\
&\leq H(\x;\x),
\end{align*}
showing that $\x \in \arg\min_{x\in X}\;  \Model_a(x;\x )+\frac{\mu^k}{2}\norm{x-\x}^2$ for all $a\in A_H(\x)$. Convexity of $\Model_a(\cdot;x)$ yields $\x \in \arg\min_{x\in X}\;  \Model_a(x;\x )$ for all $a\in A_H(\x)$, i.e., $\x$ is M-critical.

From now on, let us assume that the algorithm loops forever:

\textbf{Claim 1: Algorithm 1 produces finitely many null steps.}
Note that if $\mu^k\geq \bar \mu +\kappa$ (with $\bar \mu$ ensured by Assumption~\ref{assump-model}), then $y^k$ is declared a serious iterate. 
Indeed,
\begin{align*}
H(y^k;x^k)&\leq 
\Model(y^k;x^k)+\frac{\bar \mu}{2}\norm{y^k-x^k}^2\\
&
= \Model(y^k;x^k)+\frac{\bar \mu+\kappa}{2}\norm{y^k-x^k}^2 - {\displaystyle\frac{\kappa}{2}}\norm{y^k-x^k}^2 \\
& \leq \min_{x\in X}\col{ \Model(x;x^k)+\frac{\mu^k}{2}\norm{x-x^k}^2}+\epsilon^k - {\displaystyle\frac{\kappa}{2}}\norm{y^k-x^k}^2 \\
&\leq  \Model(x^k;x^k) +\epsilon^k - \frac{\kappa}{2}\norm{y^k-x^k}^2\\
&\leq  H(x^k;x^k) -\frac{\kappa-\lambda}{2}\norm{y^k-x^k}^2 ,
\end{align*}
where the first inequality is due to Assumption~\ref{assump-model}(b), and the two last ones follow from definition of $y^k$ and identity $\Model(x^k;x^k)=H(x^k;x^k)$. Thus $y^k$ satisfies the descent test and becomes a serious iterate.
As $\mu^k$ increases by at least $\delta>0$ after a null step (hypothesis {(ii)}), we conclude that Algorithm~\ref{algo}
generates \wlo{at most $\lceil\frac{\bar\mu+\kappa-\mu^0}{\delta}\rceil$ null steps}. Hence, there exists $\bar k $ such that, for all $k\geq \bar k$,
\[
\mu^k=\mu':= \mu^{\bar k},\quad x^{k+1}\in\epsilon^k\mbox{-}\arg\min_{x\in X} \Model(x;x^k)+\frac{\mu'}{2}\norm{x-x^k}^2,\quad\mbox{and}
\] 
\[
H(x^{k+1};x^k)\leq H(x^k;x^k) -\frac{\kappa-\lambda}{2}\norm{x^{k+1}-x^k}^2.
\]

\textbf{Claim 2: The sequence $\{\norm{y^{k}-x^k}\}_k$ vanishes.}
First, note that $x^{k+1}=y^k$ for all $k\geq \bar k$. We split our analysis into two cases: there exists an iteration counter $k_1\geq \bar k$ such that $x^{k_1}$ is feasible for problem~\eqref{pb} or such a $k_1$ does not exist. 
If we assume the existence of such $k_1$, by an inductive argument on Lemma \ref{lemma:ineq} (item {(i.1)} with $\mu=\kappa-\lambda>0$) we can deduce that for every $k\geq k_1$ the iterate $x^k\in X^c$ and the following inequality holds
    $$
    f(x^{k+1})\leq f(x^k)-\frac{\kappa-\lambda}{2}\norm{x^{k+1}-x^k}^2.
    $$
    Rearranging the terms and adding the last expression for every $k\geq k_1$, we obtain
    $$
    \sum_{k=k_1}^\infty \norm{x^{k+1}-x^k}^2\leq \displaystyle\frac{2}{\kappa-\lambda}\sum_{k=k_1}^\infty \left(\wlo{f(x^k)- f(x^{k+1})}\right)\leq \frac{2}{\kappa-\lambda}\left(f(x^{k_1})-\lim_{k\to\infty}f(x^k))\right).
    $$
    It follows by the continuity of $f$,  and boundedness of $\{x^k\}_k$  that the right-hand side of the above inequality is finite, and so is the sum at the left-hand side. This implies that  $\{\norm{x^{k+1}-x^k}\}_k$ converges to zero.
    If we now assume that $x^k$ is not feasible for every $k\geq \bar k$ then, inequality 
    $$
    c(x^{k+1})\leq c(x^{k}) -\frac{\kappa-\lambda}{2}\norm{x^{k+1}-x^k}^2,
    $$
    holds for every $k\geq \bar k$ by applying induction on Lemma \ref{lemma:ineq} (item {(i.2)}). Then, by a similar argument as the one used in the previous case, we deduce that the sequence $\{\norm{x^{k+1}-x^k}\}_{k}$ vanishes. 
Hence, $\epsilon^k \to 0$ (due to \eqref{spbm}) and $\{y^{k_i}\}_i$ also converges to $\x$:
\[
\lim_iy^{k_i}= \lim_{i} x^{k_i+1} =  \lim_{i} x^{k_i} = \x.
\]

\textbf{Claim 3: Epigraphical nesting of the sequence.}
Remark that, using hypothesis (iii), we can consider a sequence $z^i\to x$ for any given $x\in \R^n$ holding item (ii) of Lemma \ref{lemma:epi2} for $\Model(\cdot;x^{k_i})+{\rm i}_X$. Thus,
$$
 \limsup_i \Model(z^i;x^{k_i})+{\rm i}_X(z^i)+\frac{\mu^{k_i}}{2}\norm{z^i-x^{k_i}}^2\leq\bar\Model(x;\x )+{\rm i}_X(x)+\frac{\mu'}{2}\norm{x-\x}^2
$$
which, again by item (ii) in Lemma \ref{lemma:epi2}, leads to
\[
\epi{\bar\Model(\cdot;\x )+\frac{\mu'}{2}\norm{\cdot-\x }^2+{\rm i}_X}\subseteq\LimInn{\epi{\Model(\cdot;x^{k_i})+\frac{\mu^{k_i}}{2}\norm{\cdot-x^{k_i}}^2+{\rm i}_X}}.
\] 

Observe that, for all $k$,
\begin{align*}
H(y^k;x^k)&\leq 
\Model(y^k;x^k)+\frac{\bar \mu}{2}\norm{y^k-x^k}^2\\
& = \Model(y^k;x^k)+\frac{\mu^k}{2}\norm{y^k-x^k}^2 + \frac{\bar \mu-\mu^k}{2}\norm{y^k-x^k}^2\\
& \leq  \min_{y \in X}\Model(y;x^k)+\frac{\mu^k}{2}\norm{y-x^k}^2+ \epsilon^k + \frac{\bar \mu-\mu^k}{2}\norm{y^k-x^k}^2\\
& \leq  H(x^k;x^k) + \frac{\bar \mu-\mu^k+\lambda}{2}\norm{y^k-x^k}^2.
\end{align*}
Thus, for the given subsequences and taking the limit with $i$ going to infinity, the continuity of $H$ leads to
\begin{align*}
\lim_{i} \Model(y^{k_i};x^{k_i})+\frac{\mu^k}{2}\norm{y^{k_i}-x^{k_i}}^2  +\ind_X(y^{k_i})
&=H(\x;\x) \\
&= \Model_a(\x;\x)+\frac{\mu'}{2}\norm{\x-\x }^2\quad \forall\; a \in A_H(\x)\\
&= \bar \Model(\x;\x)+\frac{\mu'}{2}\norm{\x-\x }^2.
\end{align*}

\wlo{Recall that, by definition,
$$
y^k\in\epsilon^k\mbox{-}\arg\min_{y \in X}\;{\Model(y;x^k)  + \frac{\mu^k}{2}\norm{y-x^k}^2}
$$
which together with $y^k\to\x$ (due to Claim 2) and the epigraphical nesting of the sequence (Claim 3), we can conclude that
 $
 \x \in \arg\min_{x\in X}\bar \Model(x;\x )+\frac{\mu'}{2}\norm{x-\x}^2
 $ due to Propostion \ref{prop:conv}, which in turn gives}
\begin{align*}
H(\x;\x)&=\min_{x\in X}\bar \Model(x;\x )+\frac{\mu'}{2}\norm{x-\x}^2\\
&= \min_{x\in X}\; \min_{a\in A_H(\x)}\; \Model_a(x;\x )+\frac{\mu'}{2}\norm{x-\x}^2\\
&\leq  \min_{x\in X}\;  \Model_a(x;\x )+\frac{\mu'}{2}\norm{x-\x}^2\quad \forall\; a\in A_H(\x)\\
&\leq H(\x;\x),
\end{align*} 
showing that $\x \in \arg\min_{x\in X}\;  \Model_a(x;\x )+\frac{\mu'}{2}\norm{x-\x}^2$ for all $a\in A_H(\x)$. Convexity of $\Model_a(\cdot;x)$ yields $\x \in \arg\min_{x\in X}\;  \Model_a(x;\x )$ for all $a\in A_H(\x)$, i.e., $\x$ is M-critical.
In addition, if  $x^0 \in X^c=\col{x\in X:\, c(x)\leq 0}$, then the algorithm ensures that $\{x^k\}_k\subset X^c$ due to Lemma~\ref{lemma:ineq} {(i.1)}. FM-criticality of $\x$ follows from the fact that $X^c$ is closed (as $X$ is closed and $c$ is continuous). \mybox
\end{proof}

We conclude this section with a few comments on the hypothesis employed by Theorem~\ref{theo:conv}.
Note that hypothesis {(i)} is simple to satisfy (e.g. if $X$ is compact). Hypothesis {(ii)} is a trivial one. Its goal is to ensure that if the algorithm produces infinitely many null steps, then $\mu^k\to \infty$.
The more stringent conditions are the ones related to the model, i.e., Assumption~\ref{assump-model} and hypothesis {(iii)}. We recall that the threshold $\bar \mu>0$ in Assumption~\ref{assump-model} is not required to be known. In the next sections, we show how to build proper models for which Assumption~\ref{assump-model} and hypothesis {(iii)} are satisfied.

\section{Optimization Problems from Stochastic Programming}\label{sec:appl}

The following structure is motivated by reliability-based optimization problems, where the design cost of an engineering system is minimized while satisfying reliability constraints \cite{Byun_Oliveira_Royset_2023}. The structure covers chance constraints and buffered optimization models constructed via sample average approximation. Details are postponed to the section on numerical experiments. Let us focus on the resulting deterministic optimization problem, which is a particular case of problem \eqref{pb}
\begin{equation}\label{pbm:sum_max} 
\left\{
    \begin{array}{lll}
         \displaystyle \min_{x \in X}& f(x):= \displaystyle \sum_{j=1}^{N_f}\max_{\ell \in F_j}\{\f^1_{j\ell}(x)+\f^2_{j\ell}(x)\} \\
         \mbox{s.t.} & c(x):= \displaystyle \sum_{j=1}^{N_c}\max_{\ell \in C_j}\{\c^1_{j\ell}(x)+\c^2_{j\ell}(x)\}\leq 0, 
    \end{array}
    \right.
\end{equation}
under the following assumptions:
\begin{enumerate}[itemsep=2pt, topsep=2pt, label=(\alph*), leftmargin=.8cm, start=1]
\item $\emptyset \neq X$ is a convex and closed subset of $\Re^n$;
$\f^1_{j\ell},\c^1_{j\ell}: \Re^n \to \R$ are convex functions, $\f^2_{j\ell},\c^2_{j\ell}: \Re^n \to \R$ are weakly-concave functions, and $F_j$, $C_j$ are finite index sets.
\end{enumerate}
Recall that a function $\phi:\Re^n\to \Re$ is weakly-concave if there exists a modulus $\mu_\phi\geq 0$ such that $\phi(x)-\frac{\mu}{2}\norm{x}^2$ is concave for all $\mu\geq \mu_\phi$. 
This implies that,
for all $\mu\geq \mu_\phi$ and all $\xi_\phi^x \in \pclarke \phi(x)$, the following inequality holds
\begin{equation}\label{wc_ineq}
\phi(y)\leq \phi(x)+ \inner{\xi_\phi^x }{y-x}+\frac{\mu}{2}\norm{y-x}^2 \quad \forall\, y\in \Re^n.
\end{equation}

In what follows we work with the improvement function $H$ of this problem and a model $\Model$ for $H$ satisfying Assumption~\ref{assump-model} and hypothesis {(iii)} in Theorem~\ref{theo:conv}. We will work with Clarke subdifferentials, which are well defined since the finite-valued functions $f$ and $c$ in~\eqref{pbm:sum_max} are locally Lipschitz. \wlo{We do not rely on first (or higher) order} oracles for $f$ and $c$. Instead, we assume that there are individual oracles for the component functions in~\eqref{pbm:sum_max}, providing us with function values and at least one subgradient for every component function: given $x\in X$, the oracles provide us with $\f^2_{j\ell}(x)$,
$\zeta_{{j\ell}} \in \pclarke \f^2_{j\ell}(x)$, $\c^2_{j\ell}(x)$, and $\xi_{{j\ell}} \in \pclarke \c^2_{j\ell}(x)$.
With this partial information,  let us define the following convex model:
\begin{subequations}\label{mod:sum_max}
\begin{equation}\label{mod:sum_max-a}
\Model_{a}(y;x):=\max\col{
\begin{array}{ll}
\displaystyle\sum_{j=1}^{N_f}\max_{\ell \in F_j}\col{\f^1_{j\ell}(y)+  \f^2_{j\ell}(x)+\inner{ \zeta_{{j\ell}}}{y-x}}-\tau_f(x),\\
\displaystyle\sum_{j=1}^{N_c}\max_{\ell \in C_j}\col{\c^1_{j\ell}(y)+ \c^2_{j\ell}(x)+\inner{ \xi_{{j\ell}}}{y-x}}
\end{array}
}
\end{equation}
where $a$ denotes the tuple $(\zeta,\xi)$ with $\zeta=(\zeta_{j\ell})_{j\ell}$ and $\xi=(\xi_{j\ell})_{j\ell}$ being the subgradients given by the oracles.
If the oracles can provide us with  more than one tuple, say a set
\begin{equation}\label{mod:sum_max-b}
A(x)\subseteq \col{(\zeta,\xi):\zeta_{{j\ell}}\in \pclarke \f^2_{j\ell}(x)\;,\;\xi_{{j\ell}} \in \pclarke \c^2_{j\ell}(x)\quad \forall\,j,\ell}
\end{equation}
of them, we may consider the model
\begin{equation}\label{mod:sum_max-c}
    \Model(y;x)=\min_{a \in A(x)}\;\Model_a(y;x).
\end{equation}
\end{subequations}
Observe that $\Model_a(\cdot;x)$ is convex and $\Model_a(x;x)=H(x;x)$ for all $x$ and $a\in A(x)$, yielding thus Assumption~\ref{assump-model}(a).
In the remainder of this section we assume that
\begin{equation}\label{A(x)=n}
|A(x)| \leq {\tt n}< \infty, \quad \mbox{with $ {\tt n}\geq 1$ an integer number.}
\end{equation}
This means that we get at most ${\tt n}$ subgradient tuples from the oracles.
In what follows we proceed to show that $\Model$ above satisfies also Assumption~\ref{assump-model}(b) and hypothesis {(iii)} in Theorem~\ref{theo:conv}. 

 \subsection{Checking the Required Assumptions on the Model}
 The following result shows that  model $\Model$ in~\eqref{mod:sum_max} satisfies Assumption~\ref{assump-model}(b).
\begin{lemma}\label{lem:model_ineq}
Let each function $\f^2_{j\ell},\c^2_{j\ell}:\Re^n \to \R$ in \eqref{pbm:sum_max} be weakly concave.
Then there exists $\bar \mu\geq 0$ such, for all $\mu\geq \bar \mu$, 
\[
H(y;x)\leq \Model(y;x)+\frac{\mu}{2}\norm{y-x}^2\quad \forall\; y,x\in \Re^n.
\]
\end{lemma}
\begin{proof} 
Let $\mu_{j\ell}\geq 0$ be the modulus of the weakly-concave function $\f^2_{j\ell}$. It follows from~\eqref{wc_ineq} that $\f^2_{j\ell}(y)\leq\f^2_{j\ell}(x)+\inner{ \zeta_{{j\ell}}}{y-x}+\frac{\mu}{2}\norm{y-x}^2$ for all $\mu\geq \mu_{j\ell}$, independently of the choice of $\zeta$.
As every index set $F_i$ in~\eqref{pbm:sum_max} are finite and $j=1,\ldots,N_f$, the value $\bar \mu_f:=\max_{j=1,\ldots,N_f}\max_{\ell \in F_j} \mu_{j\ell}$ is finite.
Hence, for all $y  \in \R^n$ and all $\mu \geq \bar \mu_f$,
\begin{align*}
f(y)&\leq \sum_{j=1}^{N_f}\max_{\ell \in F_j}\col{ \f^1_{j\ell}(y) + \f^2_{j\ell}(x)+\inner{ \zeta_{{j\ell}}}{y-x}+\frac{\bar \mu_f}{2}\norm{y-x}^2}\\
&= \sum_{j=1}^{N_f}\max_{\ell \in F_j}\col{ \f^1_{j\ell}(y) + \f^2_{j\ell}(x)+\inner{ \zeta_{{j\ell}}}{y-x}}+\frac{N_f\bar \mu_f}{2}\norm{y-x}^2.
\end{align*}
An analogous argument concerning functions $\c^2_{j\ell}$ ensures the existence of a threshold $\bar \mu_c<\infty$ such that 
\[
c(y)\leq \sum_{j=1}^{N_c}\max_{\ell \in C_j}\col{ \c^1_{j\ell}(y) + \c^2_{j\ell}(x)+\inner{ \xi_{{j\ell}}}{y-x}}+\frac{N_c\bar \mu_c}{2}\norm{y-x}^2.
\]
Then, by taking $\bar \mu=\max\col{N_f\bar \mu_f,N_c\bar \mu_c}$ and recalling the definitions of $H$ in~\eqref{eq:IF} and $\Model_a$ in~\eqref{mod:sum_max-a}, we conclude that
$H(y;x)\leq \Model_a(y;x)+\frac{\bar \mu}{2}\norm{y-x}^2$. The result follows by taking the minimum over $a \in A(x)$ given in~\eqref{mod:sum_max-b} and recalling \eqref{mod:sum_max-c}.
  \mybox
\end{proof}

Next, we show that $\Model$ in~\eqref{mod:sum_max} also satisfies hypothesis {(iii)} in Theorem~\ref{theo:conv}.
To that end, we shall keep in mind the following observation.
\begin{remark}\label{rem:LimitingModel}
For any given cluster point $\x$, there can be several subsequences of subgradients of the same function $ \f^2_{j\ell}$ converging to different cluster points, all belonging to $\pclarke \f^2_{j\ell}(\x)$ (analogously for the function $\c^2_{j\ell}$).  As a result, there can be several different models $\bar\Model(\cdot;\x)$, any of which suffices to show that hypothesis {(iii)} in Theorem~\ref{theo:conv} is satisfied. 
Indeed,
recall that
given a bounded sequence $\{x^k\}_k$, then the sequences $\{\zeta_{{j\ell}}^k\}_k$ and $\{\xi_{{j\ell}}^k\}_k$ of  subgradients are bounded as well. 
Given an arbitrary cluster point $\x$ of $\{x^k\}_k$, we can therefore extract convergent subsequences,
$\{x^{k_i}\}_i$, $\{\zeta_{{j\ell}}^{k_i}\}_i$ and $\{\xi_{{j\ell}}^{k_i}\}_i$
such that 
\begin{equation}\label{eq:myNiceSubseq}
\lim_{i } x^{k_i} =\x,\quad
\lim_{i } \zeta_{{j\ell}}^{k_i} =\bar \zeta_{{j\ell}} \in \pclarke \f^2_{j\ell}(\x), \quad \mbox{and}\quad
\lim_{i } \xi_{{j\ell}}^{k_i} =\bar \xi_{{j\ell}} \in \pclarke \c^2_{j\ell}(\x),\quad \forall j,\ell
\end{equation}
because the Clarke subdifferential is locally bounded and outer-semicontinuous \cite[Prop. 2.1.5]{Clarke_1987}. 
This limiting tuple of subgradients defines $\bar a$: under assumption~\eqref{A(x)=n}, such tuple is associated with some $a^{k_i'}\in A(x^{k_i'})$ (in the sense that $\lim_{i}a^{k_i'}\to\bar a$) and thus we can define $A(\x)$ as $\LimOut{A(x^{k_i})}$. 
Furthermore, as each index $a \in A(x^k)$ satisfies $\Model_a(x^k;x^k)=H(x^k;x^k)$, we have that $A_H(x^k)=A(x^k)$ for all $k$. All in all, we can conclude that
\[
A_H(\bar x)=A(\bar x) := \LimOut{A(x^{k_i})}.
\]
\mytriangle
\end{remark}
Hence, we can define the following model at cluster point $\x $ by employing those cluster points of subgradients defined in \eqref{eq:myNiceSubseq}:
\begin{align}
\bar\Model(y;\x)&:=\min_{a\in  A_H(\x)}\Model_a(y;\x) \nonumber\\
&=
\min_{a\in A_H(\x)}\max\col{
\begin{array}{ll}
\displaystyle\sum_{j=1}^{N_f}\max_{\ell \in F_j}\col{\f^1_{j\ell}(y)+  \f^2_{j\ell}(\x)+\inner{\bar \zeta_{{j\ell}}}{y-\x}}-\tau_f(\x),\\
\displaystyle\sum_{j=1}^{N_c}\max_{\ell \in C_j}\col{\c^1_{j\ell}(y)+ \c^2_{j\ell}(\x)+\inner{\bar \xi_{{j\ell}}}{y-\x}}
\end{array}
}.\label{eq:LimitingModel}
\end{align}
\begin{theorem}\label{theo:instance1}
    Given model $\Model$ in \eqref{mod:sum_max} satisfying~\eqref{A(x)=n}, assume hypotheses (i) and (ii) in Theorem \ref{theo:conv} hold. Then every cluster point $\x$ of the sequence $\{x^k\}_k$ computed by Algorithm \ref{algo} is M-critical. Furthermore, provided that an iterate $x^k$ is feasible, $\x$ is FM-critical.
\end{theorem}
{\bf Proof}
Given the subsequence $\{x^{k_i}\}_i$ from \eqref{eq:myNiceSubseq} and $\bar\Model(\cdot;\x)$ as in~\eqref{eq:LimitingModel}, the result follows from Theorem~\ref{theo:conv} once show that $\epi{\bar\Model(\cdot;\x)+{\rm i}_X}\subseteq\LimInn{\epi{\Model(\cdot;x^{k_i})+{\rm i}_X}}$. To this end, consider an arbitrary point of $y\in X$ and any sequence $\{y^i\}_i\subseteq X$ such that $y^i\to y$. For the specific sequences in
\eqref{eq:myNiceSubseq} we get (see  Remark~\ref{rem:LimitingModel})
\begin{align*}
\bar\Model(y;\x) & =\min_{a \in A_H(\bar x)}\Model_a(y;\bar x)
= \min_{a \in \LimOut{A(x^{k_i})}}\Model_a(y;\bar x)\\
& = \lim_i\min_{a \in A(x^{k_i})}\Model_a(y^i;x^{k_i})=\lim_i\Model(y^i;x^{k_i}),
\end{align*}
where the third equality follows from our construction of $\LimOut{A(x^{k_i})}$ and continuity of $\Model_a(\cdot;\cdot)$ on both arguments. 
Lemma \ref{lemma:epi2} item (ii) yields the wished inclusion.
\mybox

We stress that in the above theorem, the only theoretical gain in working with $|A(x^k)|>1$ (i.e., a nonconvex model) is the fact that $A_H(\x)$ in our criticality definition \ref{def:crit} can have more than one index. Thus, the criticality we get in practice might be more robust: a richer model will likely yield better results. In Section~\ref{sec:dcset}, we consider a specific DC setting and show that a richer model yields $B$-stationary, the strongest stationary condition in DC programming.

\subsection{Particular Instances}
Let us now consider simpler variants of problem~\eqref{pbm:sum_max}.

\paragraph{Case 1:}
Consider problem~\eqref{pbm:sum_max}  with  $N_f=N_c=1$ and $F_1=C_1=\{1\}$.
In this case, $f(x)=f_1(x)+f_2(x)$, $c(x)=c_1(x)+c_2(x)$ with $f_1,c_1$ convex and $f_2,c_2$ weakly-concave functions. Moreover, let us consider $f_2$ and $c_2$ to be smooth (i.e., continuously differentiable).
Model~\eqref{mod:sum_max} reduces to 
\[
\Model(y;x)=\max\col{\begin{array}{l}
f_1(y)+f_2(x)+\inner{\nabla f_2(x)}{y-x}-\tau_f(x),\\
c_1(y)+c_2(x)+\inner{\nabla c_2(x)}{y-x}
\end{array}}
\]
as $A(x)=\{(\nabla f_2(x),\nabla c_2(x))\}$. If $\x \in X$ is FM-critical for the model $\bar\Model(\cdot;\x)=\Model(\cdot;\x)$ (which coincides with the one in \eqref{eq:LimitingModel} due to  continuity of the gradients), we would have that
\[
\x \in\arg\min_{x \in X} \; \Model(x;\x);\quad 0\geq c(\x)= [c_1(\cdot)+c_2(\x)+\inner{\nabla c_2(x)}{\cdot-\x}](\x).
\]
In that case, and assuming the Constraint Qualification (CQ) $c(\x)=0\Longrightarrow 0 \not \in \partial c_1(\x )+\nabla c_2(\x) +N_X(\x)$ holds,
 there exists $\bar u \in \R$ such that the following generalized KKT system is satisfied
\[
\left\{
\begin{array}{llllll}
0 \in \partial f_1(\x) + \nabla f_2(\x) + \bar u[\partial c_1(\x)+\nabla c_2(\x)] + N_X(\x)\\
\x \in X,\; c(\x)\leq 0,\; \bar u\geq 0,\;  \bar u c(\x)=0.
\end{array}\right.
\]
Note that, as $ \partial f_1(\x) + \nabla f_2(\x)=\pclarke f(\x)$ and $ \partial c_1(\x) + \nabla c_2(\x)=\pclarke c(\x)$ \cite[Prop. 2.3.3]{Clarke_1987}, our FM-criticality boils down to the standard KKT optimality conditions (stationarity) in nonsmooth nonlinear programming, given the stated CQ holds.

In the setting of convex plus weakly-concave functions a bundle method with convergence guarantees was recently proposed in \cite{Ksenia_2023}. 

\paragraph{Case 2 (DC setting):}
Here, not only $f_1$ and $c_1$ but also $-f_2$ and $-c_2$ are assumed to be convex functions.
Accordingly, the threshold $\bar \mu$ in Assumption~\ref{assump-model} is zero. 
Furthermore, the criticality condition is as in the previous paragraph, replacing $\nabla f_2$ and $\nabla c_2$ by $\pclarke f_2$ and $\pclarke c_2$, boiling down to the usual criticality condition in DC-constrained DC programming \cite{Pang_2017,Javal_2021}.
As we discuss next, stronger conditions are possible after adding assumptions and making changes in the model $\Model$.

\section{Special DC Setting}\label{sec:dcset}
We have just shown that DC-constrained DC programming problems fit into formulation~\eqref{pbm:sum_max}, and thus the M-criticality coincides (when ${\tt n}=1$ in \eqref{A(x)=n}) with the usual criticality concept employed in the DC literature.
The goal here is to consider a special class of DC problems, a different model, and show that Algorithm~\ref{algo} obtains the sharpest stationarity definition in DC programming, i.e., $B$-stationarity. The family of DC programs we investigate is the one considered in~\cite{Pang_2017}\wlo{, where an algorithm is presented to compute $B$-stationary points of the same family of problems}. Consider problem~\eqref{pb} with
\begin{subequations}\label{specialDC}
\begin{equation}
\mbox{
$f(x):=f_1(x)+f_2(x)$, $c(x)=c_1(x)+c_2(x)$, $f_1,c_1:\Re^n\to \Re$ convex, and 
}\end{equation}
\begin{equation}\label{f2:conc}
f_2(x):=\min_{j=1,\dots,m_f}\f_j(x)\quad\text{and}\quad c_2(x):=\min_{\ell=1,\dots,m_c}\c_\ell(x),
\end{equation}
\end{subequations}
where all $\f_j,\c_\ell:\R^n\to\R$ are continuously differentiable concave functions.
For $\varepsilon\geq 0$, let us define the following index set:
\[
A_f^\varepsilon(x):=\col{j: \f_j(x)\leq f_2(x)+\varepsilon},
\quad\mbox{and}\quad
A_c^\varepsilon(x):=\col{\ell: \c_\ell(x)\leq c_2(x)+\varepsilon}.
\]
For $\varepsilon=0$, these sets reduce to the active index sets of each function.
Let $\x \in X^c$ be a point satisfying the following \emph{pointwise Slater} CQ:  there exist $\ell\in A_c^0(\x)$, and $d^\ell \in \mathcal{T}_X(\x)$ such that $c'_1(\x;d^\ell) < \inner{\nabla \c_\ell(\x)}{d^\ell}$. 
Under this CQ, it was shown in \cite[Lem. 2.1]{Ackooij_Oliveira_2022} that $\x$ is an $B$-stationary point to problem \eqref{pbm:sum_max} under the setting of \eqref{specialDC} if and and only if $\x$ solves the following convex subproblems for all $j\in A_f^0(\x)$ and $\ell\in A_c^0(\x)$ (see also \cite[Lem. 1]{Pang_2017} for a somehow stronger requirement):
\begin{equation}\label{sbpmij}
\left\{
\begin{array}{lll}
\displaystyle \min_{x \in X} & f_1(x)+ \f_j(\x)+\inner{\nabla \f_j(\x)}{x-\x}\\
\mbox{s.t.} &  c_1(x)+\c_\ell(\x)+\inner{\nabla \c_\ell(\x)}{x-\x} \leq 0.
\end{array}
\right.
\end{equation}
With this in mind, we propose to furnish Algorithm~\ref{algo} with the following model for $H$:
\begin{equation}\label{model:dc}
\Model(y;x)=\max\col{\begin{array}{l}
f_1(y)+\displaystyle\min_{j\in A_f^\varepsilon(x)}  [\f_j(x)+\inner{\nabla \f_j(x)}{y-x}]-\tau_f(x),\\
c_1(y)+\displaystyle\min_{\ell\in A_c^\varepsilon(x)} [\c_2(x)+\inner{\nabla \c_\ell(x)}{y-x}]
\end{array}}.
\end{equation}
This nonconvex model can be equivalently written as
\begin{align}
\Model(y;x)& =\displaystyle\min_{j\in A_f^\varepsilon(x)} \displaystyle\min_{\ell\in A_c^\varepsilon(x)}  \max\col{\begin{array}{l}
f_1(y)+ [\f_j(x)+\inner{\nabla \f_j(x)}{y-x}]-\tau_f(x),\\
c_1(y)+[\c_2(x)+\inner{\nabla \c_\ell(x)}{y-x}]
\end{array}} \nonumber \\
&=\displaystyle\min_{j\in A_f^\varepsilon(x)} \displaystyle\min_{\ell\in A_c^\varepsilon(x)}  \Model_{j\ell}(y;x),\label{mod:DC_B}
\end{align}
with every $\Model_{j\ell}(\cdot;x)$ convex. Thus, we are in the setting of Assumption \ref{assump-model} with $\bar\mu=0$ and $A(x)=A^\varepsilon_f(x)\times A^\varepsilon_c(x)$. 
\begin{lemma}\label{lem:epsilon}
    Let $\varepsilon\geq 0$ and suppose that $\x$ is FM-critical. Furthermore, suppose that the pointwise Slater CQ is satisfied if $c(\x)= 0$. Then $\x$ is an $B$-stationary point to problem~\eqref{pb} having structure~\eqref{specialDC}. 
\end{lemma}
\begin{proof}
Note that model~\eqref{mod:DC_B} satisfies, for all $\varepsilon\geq 0$, Assumption \ref{assump-model} (the threshold $\bar\mu$ is zero in the DC setting). Recalling Remark \ref{remark:subproblems} and assuming $\x$ to be FM-critical
\[
\x \in \displaystyle \arg\min_{x \in X}\; \Model_{j\ell}(x;\x) \quad \forall\, j \in A_f^0(\x),\;\forall\,\ell \in A_x^0(\x)
\]
and   $c(\x)\leq 0$, yielding $\Model_{j\ell}(\x;\x)= H(\x;\x)=0$. Hence,
\[
 f_1(x)+ \f_j(\x)+\inner{\nabla \f_j(\x)}{x-\x}\geq f(\x) \quad \forall\; x \in X:\; c_1(x)+\c_\ell(\x)+\inner{\nabla \c_\ell(\x)}{x-\x}\leq 0
\]
and all $j \in A_f^0(\x)$ and $\ell \in A_x^0(\x)$.
In particular, $\x$ solves
$$
\left\{
\begin{array}{lll}
\displaystyle \min_{x \in X} & f_1(x)+ \f_j(\x)+\inner{\nabla \f_j(\x)}{x-\x}\\
\mbox{s.t.} &  c_1(x)+\c_\ell(\x)+\inner{\nabla \c_\ell(\x)}{x-\x}\leq 0.
\end{array}
\right.
$$
for all $j\in A^0_f(\x)$ and $\ell\in A^0_c(\x)$. The result follows from  \wlo{\cite[Prop. 2.8]{Ackooij_Oliveira_2022}}.
\mybox
\end{proof}

\begin{theorem}
    Consider model $\Model$ in \eqref{model:dc} and assume hypothesis (i) and (ii) in Theorem \ref{theo:conv} hold. Then every cluster point $\x$ of the sequence $\{x^k\}_k$ computed by Algorithm \ref{algo} is M-critical. Furthermore, if $\x$ is feasible and the pointwise Slater CQ holds, then $\x$ is B-stationary.
\end{theorem}
{\bf Proof}
The result will follow from Theorem  \ref{theo:conv} once hypothesis (iii) is proven. To this end, let $\{x^{k_i}\}_i$ be a subsequence of $\{x^k\}_k$ such that $\lim_ix^{k_i}=\x$. Observe that continuity of \wlo{$\f_j$ and $\c_\ell$}, and $\varepsilon>0$ fixed in the model's definition ensure the existence of a neighborhood $\mathcal{N}$ of $\x $ such that $A_f^0(\x) \subset A_f^\varepsilon(y)$ and  $A_c^0(\x) \subset A_c^\varepsilon(y)$ for all $y \in \mathcal{N}\cap X$ (such set has nonempty relative interior as $\x$ belongs to the interior of $\mathcal N$ and $X$ is convex).
 Therefore, for $i$ large enough, 
 \begin{align*}
 \Model(y;x^{k_i}) &= \displaystyle\min_{j\in A_f^\varepsilon(x^{k_i})} \displaystyle\min_{\ell\in A_c^\varepsilon(x^{k_i})}  \Model_{j\ell}(y;x^{k_i}) 
  \leq \displaystyle\min_{j\in A_f^0(\x)} \displaystyle\min_{\ell\in A_c^0(\x)}  \Model_{j\ell}(y;x^{k_i}).
 \end{align*}
 Let $x \in X$ be arbitrary, and $\{y^i\}_i\subseteq X$ such that $y^i\to x$. By replacing $y$ with $y^i$ above and passing to the limit superior with $i$ going to infinity we get, due to continuity of $\Model_{j\ell}$ in both arguments,
 \[
 \limsup_i  \Model(y^i;x^{k_i})\leq \displaystyle\min_{j\in A_f^0(\x)} \displaystyle\min_{\ell\in A_c^0(\x)}  \limsup_i  \Model_{j\ell}(y^i;x^{k_i})
 = \displaystyle\min_{j\in A_f^0(\x)} \displaystyle\min_{\ell\in A_c^0(\x)}    \Model_{j\ell}(x;\x)
 =\bar\Model(x;\x).
 \]
 Noticing that $A_H(\x)=A^0_f(\x)\times A_c^0(\x)$ in Definition \ref{def:crit}, and recalling item (ii) in Lemma \ref{lemma:epi2}, the result follows.
Under the extra assumption of $\x$ being feasible and holding Slater CQ, we conclude from Lemma \ref{lem:epsilon}.
\mybox

\section{Optimization Problems from Classification Applications}\label{sec:comp}
In this section, we consider problem~\eqref{pb} with functions given by
\begin{equation}\label{eq:pbcomp}
f(x):=  \sum_{j=1}^{N_f} G(\f^1_{j}(x)+\f_{j}^2(x)), \quad
c(x):=  \sum_{j=1}^{N_c} R(\c^1_{j}(x)+\c^2_{j}(x)),
\end{equation}
with $\f^1_{j},\c^1_{j}:\Re^n\to \Re$ being convex functions, $\f^2_{j},\c^2_{j}:\Re^n\to \Re$ being concave, and
$G,R:\Re\to \Re$ being convex and nondecreasing one-dimensional functions.
Problems of this form appear in machine learning applications such as classification \cite[\S 6.I]{RoysetWets_2022}.   

Given $\zeta_{j}\in \pclarke \f^2_j(x)$ and $\xi_{j}\in \pclarke \c^2_j(x)$, we define the following model
\begin{equation}\label{mod:comp}
\Model_a(y;x):=\max\col{
\begin{array}{lll}
 \sum_{j=1}^{N_f} G(\f^1_{j}(y)+\f_j^2(x)+\inner{\zeta_{j}}{y-x}) -\tau_f(x)\\
 \sum_{j=1}^{N_c} R(\c^1_{j}(y)+\c_j^2(x)+\inner{\xi_{j}}{y-x}) 
\end{array}
},
\end{equation}
where $a\in A(x)\subseteq \col{(\zeta,\xi):\zeta_{{j}}\in \pclarke \f^2_{j}(x)\;,\;\xi_{{j}} \in \pclarke \c^2_{j}(x)\mbox{ for all }j}$.
Thus, we consider the model $\Model(y;x)=\min_{a\in A(x)}\Model_a(y;x)$ by assuming that $1\leq |A(x)|\leq {\tt n}$ (cf. condition~\eqref{A(x)=n}). We see that Assumption \ref{assump-model}(a) holds. Note also that, as $\f^1_{j}(y)+\f^2_{j}(y)\leq\f^1_{j}(y)+\f_j^2(x)+\inner{\zeta_{j}(x)}{y-x}$ and $G$ is nondecreasing (equivalently for $c$ and $R$), Assumption \ref{assump-model}(b) holds for $\bar\mu=0$.

\begin{remark}\label{rem:LimitingModel2}
As we clarified in Remark \ref{rem:LimitingModel}, given a cluster point $\x$, we can always ensure \eqref{eq:myNiceSubseq} for the sequences of subgradients composing model \eqref{mod:comp}. Considering $\bar\zeta_{j}\in \pclarke \f^2_j(\x)$ and $\bar\xi_{j}\in \pclarke \c^2_j(\x)$, the corresponding cluster subgradient vectors for every $j$, we define $A(\x)$ to be the set of such cluster points provided by the subsequences $\{\zeta_j^{k_i}\}_i$ and $\{\xi_j^{k_i}\}_i$:
$$
\bar\Model(x;\x)=\min_{a\in A_H(\x)}\max\col{
\begin{array}{lll}
 \sum_{j=1}^{N_f} G(\f^1_{j}(x)+\f_2(\x)+\inner{\bar\zeta_j}{x-\x}) -\tau_f(x)\\
 \sum_{j=1}^{N_c} R(\c^1_{j}(x)+\c_2(\x)+\inner{\bar\xi_{j}}{x-\x}) 
\end{array}
}.
$$
As in Remark~\ref{rem:LimitingModel}, we can take $A_H(\x)=A(\x)=\LimOut{A(x^{k_i})}$ for the subsequence $\{x^{k_i}\}_i$ converging to $\x$.
\mytriangle\end{remark}

\begin{theorem}\label{th:comp}
Consider Algorithm~\ref{algo} furnished with model~\eqref{mod:comp} applied to problem~\eqref{pb} with  $f$ and $c$ given in \eqref{eq:pbcomp}. Assume that the sequence $\{x^k\}_k$ computed by the algorithm is bounded. Then, any cluster point of $\{x^k\}_k$ is M-critical.
\end{theorem}
{\bf Proof}
As in the proof of Theorem~\ref{theo:instance1}, we see that 
$
\lim_i\Model(y^i;x^{k_i})=\bar\Model(y;\x)
$
for any sequence $\{y^i\}_i\subseteq X$ such that $y^i\to y$. By relying on Lemma \ref{lemma:epi2} item (ii) we get hypothesis (iii) in Theorem \ref{theo:conv}. Hence, this last theorem gives the stated result.
\mybox

\section{Numerical Experiments}\label{sec:num}
This section presents some preliminary numerical experiments \wlo{pertaining to the model discussed in Section \ref{sec:appl}} to assess the practical performance of our approach. We investigate two different (classes of) stochastic problems fitting the setting in \eqref{pbm:sum_max}.
A chance-constrained model is considered for a gas network problem from \cite{LukMartHolgRene_2018}, and a buffered probability optimization model for the design of a cantilever beam-bar system studied in \cite{Song_2003} and \cite{Byun_Oliveira_Royset_2023}.

Our numerical experiments are performed in {\tt MATLAB R2022b} on a computer Intel Core i9 2.50GHz under Windows 11 Pro.
The convex optimization solver called at every iteration of our approach is the bundle algorithm described in Appendix~\ref{ap:bundle}. 
The algorithm's parameters are set differently for each family of problems.

\subsection{Chance-Constrained Optimization Model}

Let us consider that a random vector ${\boldsymbol\omega}$ belongs in a certain probability space. Given a function $G$ depending on both $x\in\R^n$ and $\boldsymbol \omega$, we say that $x$ is a success for a given scenario $\omega$ if $G(x;\omega)\leq 0$ and it is a failure otherwise. The uncertainty arising from this problem may be dealt with the \textit{chance constraint}
$$
\mathbb{E}[{\pmb 1}_{(0,+\infty)}(G(x;{\boldsymbol \omega}))]=\mathbb{P}[G(x;{\boldsymbol\omega})> 0]\leq \alpha
$$
for a given $\alpha\in(0,1)$, which means that the probability of failure for the decision variable $x$ is allowed to be, at most, the threshold $\alpha$. Chance constraints arise in many real-life optimization problems, such as electricity network expansion, energy management, mineral blending, structural systems and other operating engineering systems  \cite{vanAckooij_Henrion_Moller_Zorgati_2011b,vanAckooij_Berge_Oliveira_Sagastizabal_2017,Prekopa_2003,Rockafellar_Royset_2010}. The probabilistic constraint is traditionally faced through the \textit{Sample Average Approximation (SAA)} (see \cite[\S  5]{SDR09}), which leads to deterministic constraint as follows
\begin{equation}\label{eq:prob}
\sum_{j=1}^Np_j{\pmb 1}_{(0,+\infty)}[G(x;\omega^j)]\leq \alpha,
\end{equation}
where $\{\omega^j\}_{j=1}^N$ is a set of $N$ realizations of the random vector ${\boldsymbol \omega}$, and $p_j$ is the probability of scenario $\omega^j$. It is clear that the above constraint does not fit \eqref{pbm:sum_max} as ${\pmb 1}_{(0,+\infty)}$ is not even a continuous function. Nevertheless, this issue will be overcome using a well-known nondecreasing $\mathcal C^2$ approximation of ${\pmb 1}_{(0,+\infty)}$, $\psi_\theta:\R\to(0,1)$ for $\theta>0$, given by

$$
\psi_\theta(x)=\frac{1}{1+\exp\left(-\frac{x}{\theta}\right)}.
$$
 In the following example, we will take $\theta=0.1$ and a sample with $N=10000$ scenarios.

\paragraph{Gas problem:} \label{pb:gas}

In this example, we consider an $n$-node connected tree (a graph with no loops) representing a gas exit network as described in \cite{ClauHolRene_2016,LukMartHolgRene_2018} (see Figure \ref{fig:trees}).  The origin node (labeled as $0$) is the one injecting the gas that flows through the pipes to the other nodes, where the gas exits with a certain pressure. The exit load of each node $\ell$ is random, and for each edge, a pressure drop is attributed. The goal is to find a reliable upper bound for the pressure flowing out of each node, while minimizing the sum of them.  Let $\omega^j$ be a realization of the random vector $\boldsymbol{\omega}$ representing the exit loads. A vector $x \in \R^n$  is considered a success  (for such a realization $\omega^j$) if the inequalities
\begin{equation}\label{eq:alpha}
    G_\ell(x,{\omega^j}):=v_\ell({\omega^j})-x_\ell^2\leq 0
\end{equation}
hold for all $\ell=0,\dots,n-1$, where $v_0(\boldsymbol{\omega})=1+\max_{\ell=1,\dots,n-1}\{h_\ell(\boldsymbol{\omega})\}$, $v_\ell(\boldsymbol{\omega})=v_0(\boldsymbol{\omega})-h_\ell(\boldsymbol{\omega})$ for $\ell=1,\dots,n-1$ and $h_\ell(\boldsymbol{\omega})$ is a function of the random vector $\boldsymbol{\omega}$ (see \cite{ClauHolRene_2016,LukMartHolgRene_2018}).

\begin{figure}
    \centering
\begin{minipage}{3cm}
   \begin{tikzpicture}
\node at (0,0) (n1) {};
\node at (0,-1) (n2) {};
\node at ({-sqrt(3)/2},{-3/2}) (n3) {};
\node at ({sqrt(3)/2},{-3/2}) (n4) {};

\filldraw[black] (n1) circle (3pt);
\filldraw[color=black,very thick,fill=white] (n2) circle (3pt);
\filldraw[color=black,very thick,fill=white] (n3) circle (3pt);
\filldraw[color=black,very thick,fill=white] (n4) circle (3pt);

\path[black,very thick] (n1)edge(n2);
\path[black,very thick] (n2)edge(n3);
\path[black,very thick] (n2)edge(n4);
\end{tikzpicture}
\end{minipage} \hspace{2cm} \begin{minipage}{3cm}
   \begin{tikzpicture}
\node at (0,0) (n1) {};

\node at (0,1) (n2) {};
\node at ({sqrt(3)/2},{1/2+1}) (n5) {};
\node at ({-sqrt(3)/2},{1/2+1}) (n6) {};
\node at ({-sqrt(3)},{2}) (n7) {};

\node at ({-sqrt(3)/2},{-1/2}) (n3) {};
\node at ({-sqrt(3)},{-1}) (n8) {};
\node at ({-sqrt(3)},{-2}) (n9) {};
\node at ({-3*sqrt(3)/2},{-1/2}) (n10) {};

\node at ({sqrt(3)/2},{-1/2}) (n4) {};
\node at ({sqrt(3)/2},{-3/2}) (n11) {};
\node at ({sqrt(3)},{0}) (n12) {};

\filldraw[black] (n1) circle (3pt);
\filldraw[color=black,very thick,fill=white] (n2) circle (3pt);
\filldraw[color=black,very thick,fill=white] (n3) circle (3pt);
\filldraw[color=black,very thick,fill=white] (n4) circle (3pt);
\filldraw[color=black,very thick,fill=white] (n5) circle (3pt);
\filldraw[color=black,very thick,fill=white] (n6) circle (3pt);
\filldraw[color=black,very thick,fill=white] (n7) circle (3pt);
\filldraw[color=black,very thick,fill=white] (n8) circle (3pt);
\filldraw[color=black,very thick,fill=white] (n9) circle (3pt);
\filldraw[color=black,very thick,fill=white] (n10) circle (3pt);
\filldraw[color=black,very thick,fill=white] (n11) circle (3pt);
\filldraw[color=black,very thick,fill=white] (n12) circle (3pt);

\path[black,very thick] (n1)edge(n2);
\path[black,very thick] (n2)edge(n5);
\path[black,very thick] (n2)edge(n6);
\path[black,very thick] (n6)edge(n7);

\path[black,very thick] (n1)edge(n3);
\path[black,very thick] (n3)edge(n8);
\path[black,very thick] (n8)edge(n9);
\path[black,very thick] (n8)edge(n10);

\path[black,very thick] (n1)edge(n4);
\path[black,very thick] (n4)edge(n11);
\path[black,very thick] (n4)edge(n12);

\end{tikzpicture}
\end{minipage}
    \caption{The left-hand subfigure represents a graph with $n=4$ nodes for the gas problem. The right-hand subfigure represents a graph with $n=12$ nodes.}
    \label{fig:trees}
\end{figure}
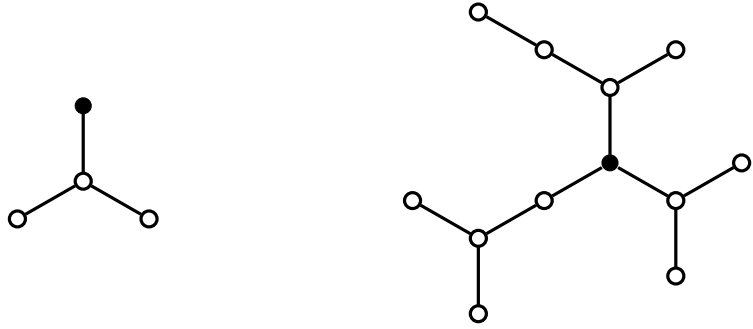 

The considered chance-constrained model is as follows
$$
\left\{
\begin{array}{lllll}
    \underset{x\geq 1}{\min} & \sum_{\ell=0}^{n-1} \;x_\ell\\
    \text{s.t.}& \displaystyle\sum_{j=1}^N\frac{1}{N}\max\left\{\psi_\theta\big( G_0(x,\omega^j)\big),\dots,\psi_\theta\big( G_{n-1}(x,\omega^j)\big)\right\}\leq \alpha,
\end{array}
\right.
$$
where the set of values $\{h_\ell(\omega^j)\}_{j,\ell}$ have been provided by  \cite{LukMartHolgRene_2018}. Thus, the problem  fits \eqref{pbm:sum_max} as $X=\{x\in\R^n:x\geq 1\}$ is a convex closed set. 
We solve it employing the model described in \eqref{mod:sum_max} and compare our results against those obtained from the DC approach of \cite{Javal_2021}. This reference
approximates the problem by a DC-constrained optimization model and applies the DC bundle solver {\tt PBMDC2}. 

Table~\ref{tab:gas} reports the obtained results for two instances with $n=4$ and $n=12$ nodes. The following parameter values were set in Algorithm~\ref{algo}:  $\kappa=0.3$, $\lambda=0.1$, $\mu_0=2$ and $\tol=10^{-6}$.

\begin{table}[htb]
\centering
\begin{tabular}{ccccccccc}
$n$  & \multicolumn{2}{c}{Iter} & SS    & \multicolumn{2}{c}{CPU} &  
     $c(\x )$             & \multicolumn{2}{c}{$f(\x)$}    \\ \hline \hline
     &\tiny{Algorithm \ref{algo}}       & \tiny{PBMDC2}      &       & \tiny{Algorithm \ref{algo}}      & \tiny{PBMDC2}     &                         & \tiny{Algorithm \ref{algo}}      & \tiny{PBMDC2}     \\ \cline{2-3} \cline{5-6} \cline{8-9}                                           $4$  & 3612       & 190        & 3600  & 113s       & 53s       & $-4.33\times 10^{-2}$             & $739.811$  & $738.797$       \\ 
$12$ & 18185      & 956         & 18184 & 11min      & 9min      & $-1.26\times 10^{-4}$ & $3146.584$ & $3142.05$                                              \\ \hline                                          
\end{tabular}
\caption{Computed results for the Gas Problem for $n=4$ and $n=12$ compared to those provided by Algorithm {\tt PBMDC2} from \cite{Javal_2021}. The Iter and SS values show the number of iterations and serious steps computed. Here, $\x$ is the last iterate computed by the algorithms.}
\label{tab:gas}
\end{table}

We have run {\tt PBMDC2} with the same parameters described in~\cite{Javal_2021}. We refrain from comparing our algorithm against the approach of \cite{LukMartHolgRene_2018} since it has been shown in \cite{Javal_2021} that  {\tt PBMDC2} performs better for this problem. Although our method is more general, Table \ref{tab:gas} shows that it provides good results compared to a more specialized algorithm.

\subsection{Buffered Probability Optimization Model}
We consider a buffered failure probability constraint on an integrable random variable ${\boldsymbol \xi}$ and recall that, for $\alpha \in (0,1)$ one has \cite{Rockafellar_Royset_2010}:  
\[
\mbox{
buffered probability of ${\boldsymbol \xi} > 0$ is no greater than $1-\alpha$  if and only if    $ {\tt AVaR}_{\alpha}({\boldsymbol \xi})\leq 0$,
}
\]
where ${\tt AVaR}_{\alpha}({\boldsymbol \xi})$  is the \emph{Average value-at-Risk} of ${\boldsymbol \xi}$, given by 
\[
\min_{t\in \Re}\; t +\frac{1}{1-\alpha}\E[\max\col{0,{\boldsymbol \xi}-t}].
\]
For further details about buffered failure probabilities, we refer to \cite[\S\, 3.E]{RoysetWets_2022} and \cite{Mafusalov_Uryasev_2018}.
%
%

Let us consider that ${\boldsymbol \xi}=\max_{\ell=1,\dots,m}\c_\ell(y;{\boldsymbol\omega})$, with $\c_j(\cdot;{\boldsymbol\omega})$ being a weakly-concave function.
After sampling according to the distribution of ${\boldsymbol\omega}$, as we did in \eqref{eq:prob}, we obtain the following expression for the constraint   $ {\tt AVaR}_{\alpha}({\boldsymbol \xi})\leq 0$
$$
c(y,t)=-t\frac{\alpha }{1-\alpha} + \sum_{j=1}^N \frac{p_j}{1-\alpha}\;\max\{t,\c_1(y,\omega^j),\dots,\c_m(y,\omega^j)\}\leq 0.
$$
Denoting $c_1(y,t)=-t\frac{\alpha }{1-\alpha}$ and $c_2(y,t)=\sum_{j=1}^N \frac{p_j}{1-\alpha}\;\max\{t,\c_1(y,\omega^j),\dots,\c_m(y,\omega^j)\}$, we end up in a suitable expression for problem \eqref{pbm:sum_max}, with no need of a second layer of approximation. 
It is worth noting that $ {\tt AVaR}_{\alpha}({\boldsymbol \xi})$  operator possesses a relevant characteristic of being level-bounded on $X\times \R$.
This ensures that the sequence generated by the algorithm is bounded. 
\paragraph{Design of cantilever beam-bar system:}\label{pb:beam}

This example investigates an optimal design of a cantilever beam-bar system illustrated in Figure~\ref{fig:cantilever}. \wlo{For a detailed explanation of how this example is constructed, see \cite[Ex. 3]{Song_2003}}. The system consists of an ideally plastic
cantilever beam of moment capacity $ \boldsymbol{\omega}_M+y_M$ and length $2L$, with $L=5$, which is
propped by an ideally rigid-brittle bar of strength $\boldsymbol{\omega}_T+y_T$. The structure
is subjected to load $\boldsymbol{\omega}_P$ that is applied in the middle of the bar. 
The moment capacity, strength, and load are assumed to be Gaussian random variables: 
$\boldsymbol{\omega}_M + y_M \sim N(y_M,300)$,
$\boldsymbol{\omega}_T + y_T \sim N(y_T,20)$, and
$\boldsymbol{\omega}_P \sim N(150,30)$. 
 In this example, we optimize the two
design variables $y_M \in [500, 1500]$ and $y_T \in [50, 150]$, representing the expected values of the moment capacity and the strength, respectively. The cost function is $2y_M+y_T$.
Then, the system can be described using the following limit-state functions determining whether a component fails or survives:
\begin{align*}
    g_1(y,\boldsymbol{\omega})  := & -(y_T + \boldsymbol{\omega}_T) + \frac{5}{16}\boldsymbol{\omega}_P,\\
    g_2(y,\boldsymbol{\omega})  := & \vphantom{\frac11} -(y_M + \boldsymbol{\omega}_M) + L\boldsymbol{\omega}_P,\\
    g_3(y,\boldsymbol{\omega})  := & -(y_M + \boldsymbol{\omega}_M) + \frac{3L}{8}\boldsymbol{\omega}_P,\\
    g_4(y,\boldsymbol{\omega})  := & -(y_M + \boldsymbol{\omega}_M) + \frac{L}{3}\boldsymbol{\omega}_P,\\
    g_5(y,\boldsymbol{\omega})  := & \vphantom{\frac11} -(y_M + \boldsymbol{\omega}_M) - 2L(y_T + \boldsymbol{\omega}_T) + L\boldsymbol{\omega}_P.
\end{align*}

\begin{figure}
    \centering
    \includegraphics[scale=0.5]{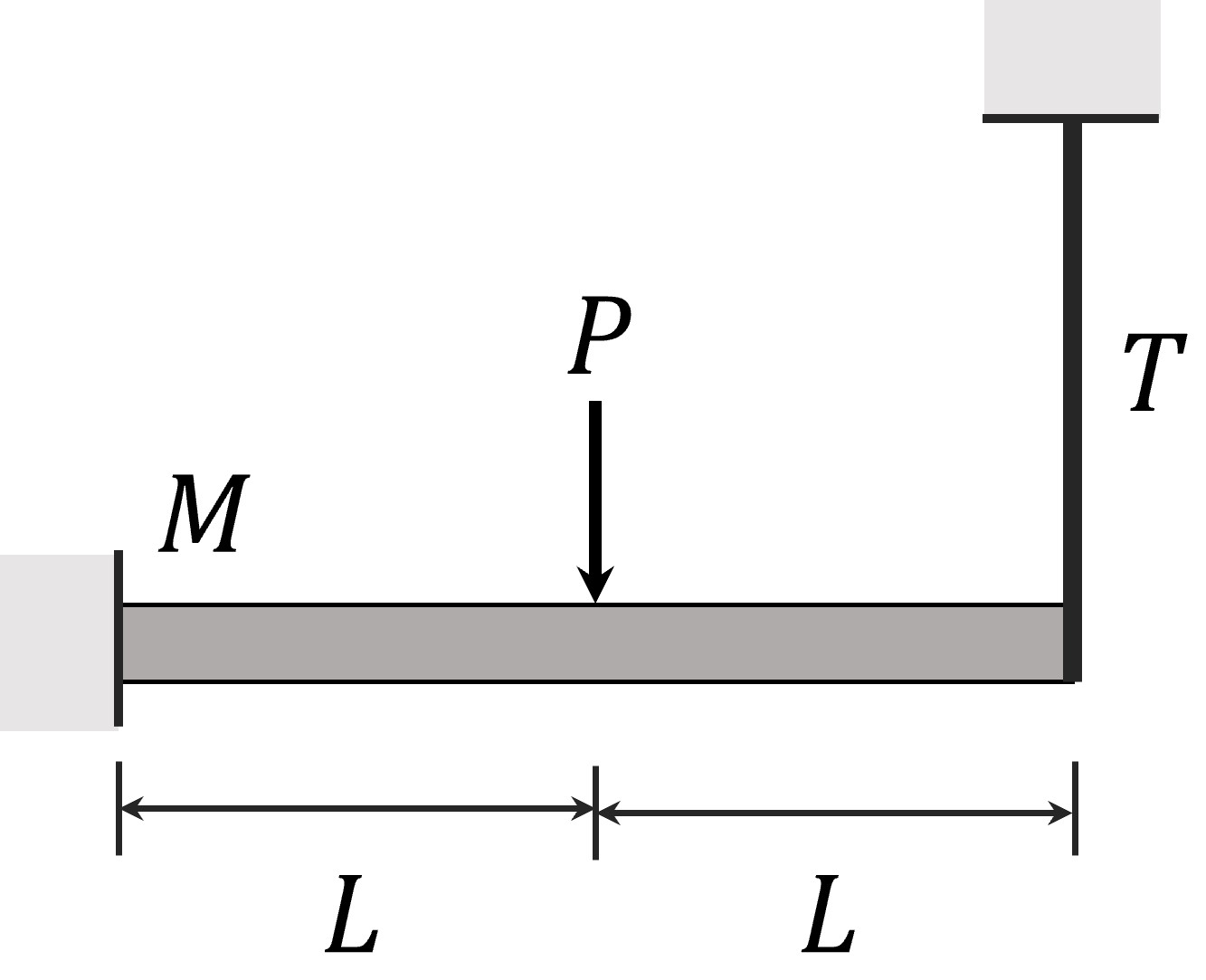}
    \caption{Example cantilever beam-bar system. Figure extracted from \cite{Byun_Oliveira_Royset_2023}.}
    \label{fig:cantilever}
\end{figure}

The system's failure can be thus represented as $\max\{G_1(y,{\omega^j}),G_2(y,{\omega^j}),G_3(y,{\omega^j})\}> 0$, with
\begin{align*}
    G_1(y,\boldsymbol{\omega})  := &\min\{g_1(y,\boldsymbol{\omega}),g_2(y,\boldsymbol{\omega})\},\\
    G_2(y,\boldsymbol{\omega})  := & \min\{g_3(y,\boldsymbol{\omega}),g_4(y,\boldsymbol{\omega})\},\\
    G_3(y,\boldsymbol{\omega})  := & \min\{g_3(y,\boldsymbol{\omega}),g_5(y,\boldsymbol{\omega})\}.
\end{align*}
Our buffered probability model for the problem  reads as
$$
\left\{\begin{array}{llll}
    \underset{(y,t)\in X}{\min} & 2y_M + y_T\\
    \text{s.t.} & \displaystyle -t\frac{\alpha}{1-\alpha} + \sum_{j=1}^N\frac{1}{N(1-\alpha)}\max\left\{ G_1(y,\omega^j), G_2(y,\omega^j),G_3(y,\omega^j),t\right\}\leq 0,
\end{array}\right.
$$
where $\omega^j$ is a realization of the random vector $\boldsymbol{\omega}=(\boldsymbol{\omega}_M,\boldsymbol{\omega}_T,\boldsymbol{\omega}_P)$, $y=(y_M,y_T)$ and  
$$
X=\left\{(y,t)\in\R^3:\begin{array}{l}
    500\leq y_M \leq 1500\\
    50\leq y_T \leq 150
\end{array}\right\}.
$$

We consider the model described in \eqref{mod:sum_max} to solve problems such as \eqref{pbm:sum_max}. For this problem, we have considered a sample with $N=100000$ scenarios and $\alpha=0.999$. The algorithm's parameters were set as $\kappa=0.3$, $\lambda=0.1$, $\mu_0=\kappa$ and $\tol=10^{-6}$. The algorithm stops after 35 seconds and 183 iterations (all of them Serious Steps). The computed feasible point $(1288,150)$ has a cost of $2727$, which is slightly smaller (under $0.1$\% of relative error) than the reference value $2729$ for the feasible point $(1289,150)$, issued by a grid search on the set $X$ with $1000\times 100$ points. Such grid search took about 240 seconds of computation.

\wlo{
\section{Conclusions}\label{sec:conc}
This paper presents an algorithmic framework for broad classes of nonsmooth and nonconvex problems using an improvement function. We avoid penalization and the associated tuning of parameters, and instead rely on subproblems minimizing an improvement function. Under mild conditions, we prove convergence of the resulting algorithm to an FM-critical point and discuss the numerous strategies for building an approximate model for improvement functions. Moreover, the inner subproblem of the algorithm can be approximately solved, for which we also suggest a bundle-like algorithm in Appendix \ref{ap:bundle}. In a numerical section, we demonstrate that the algorithm achieves satisfactory results compared to more specialized methods. Future research directions include a study of rates of convergence, analysis of algorithmic parameters, and more extensive numerical testing.
}
\begin{acknowledgements}
The second author acknowledges financial support from the Gaspard-Monge program for Optimization and Operations Research (PGMO) project ``SOLEM - Scalable Optimization for Learning and Energy Management".
The third author acknowledges support from the Office of Naval Research under grant N00014-24-1-2318.
\end{acknowledgements}


\begin{thebibliography}{}

\bibitem[Adam et~al., 2018]{LukMartHolgRene_2018}
Adam, L., Branda, M., Heitsch, H., and Henrion, R. (2018).
\newblock Solving joint chance constrained problems using regularization and
  benders' decomposition.
\newblock {\em Annals of Operations Research}, 292:683--709.

\bibitem[An and Tao, 2005]{Pham_2005}
An, L. T.~H. and Tao, P.~D. (2005).
\newblock The {DC} (difference of convex functions) programming and {DCA}
  revisited with {DC} models of real world nonconvex optimization problems.
\newblock {\em Annals of Operations Research}, 133(1):23--46.

\bibitem[Apkarian et~al., 2008]{Apkarian2008}
Apkarian, P., Noll, D., and Rondepierre, A. (2008).
\newblock {Mixed $H_2\slash H_{\infty}$ Control via Nonsmooth Optimization}.
\newblock {\em SIAM Journal on Control and Optimization}, 47(3):1516--1546.

\bibitem[Aragón-Artacho et~al., 2024]{AragonPerezMordukhovich_2023}
Aragón-Artacho, F., Mordukhovich, B., and Pérez-Aros, P. (2024).
\newblock Coderivative-based semi-{N}ewton method in nonsmooth difference
  programming.
\newblock {\em Accepted in Math. Program.}

\bibitem[Aragón-Artacho et~al., 2023]{AragonPerezTorregrosa_2023}
Aragón-Artacho, F., Pérez-Aros, P., and Torregrosa-Belén, D. (2023).
\newblock The boosted double-proximal subgradient algorithm for nonconvex
  optimization.
\newblock Submitted in July 2023.

\bibitem[Atenas et~al., 2023]{atenas_2021}
Atenas, F., Sagastizabal, C., Silva, P.~J., and Solodov, M. (2023).
\newblock A unified analysis of descent sequences in weakly convex
  optimization, including convergence rates for bundle methods.
\newblock {\em SIAM Journal on Optimization}, 33(1):89--115.

\bibitem[Burke, 1985]{Burke_1985}
Burke, J.~V. (1985).
\newblock Descent methods for composite nondifferentiable optimization
  problems.
\newblock {\em Mathematical Programming}, 33(3):260--279.

\bibitem[Burke and Ferris, 1995]{Burke_1995}
Burke, J.~V. and Ferris, M.~C. (1995).
\newblock A {G}auss---{N}ewton method for convex composite optimization.
\newblock {\em Mathematical Programming}, 71(2):179--194.

\bibitem[Byun et~al., 2023]{Byun_Oliveira_Royset_2023}
Byun, J.-E., {de Oliveira}, W., and Royset, J.~O. (2023).
\newblock {S-BORM}: Reliability-based optimization of general systems using
  buffered optimization and reliability method.
\newblock {\em Reliability Engineering \& System Safety}, 236:109314.

\bibitem[Clarke, 1987]{Clarke_1987}
Clarke, F. (1987).
\newblock {\em Optimisation and Nonsmooth Analysis}.
\newblock Classics in Applied Mathematics. Society for Industrial and Applied
  Mathematics.

\bibitem[Correa and Lemar\'{e}chal, 1993]{Correa_Lemarechal_1993}
Correa, R. and Lemar\'{e}chal, C. (1993).
\newblock Convergence of some algorithms for convex minimization.
\newblock {\em Mathematical Programming}, 62(2):261--275.

\bibitem[Cui et~al., 2022]{Cui_Pang_Liu_2022}
Cui, Y., Liu, J., and Pang, J.-S. (2022).
\newblock Nonconvex and nonsmooth approaches for affine chance-constrained
  stochastic programs.
\newblock {\em Set-Valued and Variational Analysis}, 30(3):1149--1211.

\bibitem[Cui and Pang, 2022]{Cui_Pang_2022}
Cui, Y. and Pang, J.-S. (2022).
\newblock {\em Modern Nonconvex Nondifferentiable Optimization}.
\newblock SIAM.

\bibitem[Cui et~al., 2018a]{Cui_Pang_Sen_2018}
Cui, Y., Pang, J.-S., and Sen, B. (2018a).
\newblock Composite difference-max programs for modern statistical estimation
  problems.
\newblock {\em SIAM Journal on Optimization}, 28(4):3344--3374.

\bibitem[Cui et~al., 2018b]{Cui_2018}
Cui, Y., Pang, J.-S., and Sen, B. (2018b).
\newblock Composite difference-max programs for modern statistical estimation
  problems.
\newblock {\em SIAM Journal on Optimization}, 28(4):3344--3374.

\bibitem[{de Oliveira}, 2019]{Oliveira_2019}
{de Oliveira}, W. (2019).
\newblock Proximal bundle methods for nonsmooth {DC} programming.
\newblock {\em Journal of Global Optimization}, 75(2):523--563.

\bibitem[{de Oliveira} and Solodov, 2020]{Oliveira_Solodov_2020}
{de Oliveira}, W. and Solodov, M. (2020).
\newblock {\em Bundle Methods for Inexact Data}, pages 417--459.
\newblock Springer International Publishing, Cham.

\bibitem[Diaz and Grimmer, 2023]{DiazGrimmer_2023}
Diaz, M. and Grimmer, B. (2023).
\newblock Optimal convergence rates for the proximal bundle method.
\newblock {\em SIAM Journal on Optimization}, 33(2):424--454.

\bibitem[Drusvyatskiy et~al., 2021]{druvyatskiy_2021}
Drusvyatskiy, D., Ioffe, A., and Lewis, A. (2021).
\newblock Nonsmooth optimization using taylor-like models.
\newblock {\em Math. Program.}, pages 357--383.

\bibitem[Frangioni, 2020]{Frangioni_2020}
Frangioni, A. (2020).
\newblock {\em Standard Bundle Methods: Untrusted Models and Duality}, pages
  61--116.
\newblock Springer International Publishing, Cham.

\bibitem[Gotzes et~al., 2016]{ClauHolRene_2016}
Gotzes, C., Heitsch, H., and Henrion, R. (2016).
\newblock On the quantification of nomination feasibility in stationary gas
  networks with random load.
\newblock {\em Mathematical Methods of Operational Research}, 84:427--457.

\bibitem[Higle and Sen, 1992]{HigleSen_1992}
Higle, J. and Sen, S. (1992).
\newblock On the convergence of algorithms with implications for stochastic and
  nondifferentiable optimization.
\newblock {\em Mathematics of operations research}, 17(1):112--131.

\bibitem[Le~Thi et~al., 2012]{Pham_Penalization_2012}
Le~Thi, H.~A., Pham~Dinh, T., and Ngai, H.~V. (2012).
\newblock Exact penalty and error bounds in {DC} programming.
\newblock {\em Journal of Global Optimization}, 52(3):509--535.

\bibitem[Lewis and Wright, 2016]{Lewis_2016}
Lewis, A.~S. and Wright, S.~J. (2016).
\newblock A proximal method for composite minimization.
\newblock {\em Mathematical Programming}, 158(1):501--546.

\bibitem[Liu et~al., 2020]{Liu_Cui_2020}
Liu, J., Cui, Y., Pang, J.-S., and Sen, S. (2020).
\newblock Two-stage stochastic programming with linearly bi-parameterized
  quadratic recourse.
\newblock {\em SIAM Journal on Optimization}, 30(3):2530--2558.

\bibitem[Mafusalov and Uryasev, 2018]{Mafusalov_Uryasev_2018}
Mafusalov, A. and Uryasev, S. (2018).
\newblock Buffered probability of exceedance: Mathematical properties and
  optimization.
\newblock {\em SIAM Journal on Optimization}, 28(2):1077--1103.

\bibitem[Montonen and Joki, 2018]{Montonen_Joki_2018}
Montonen, O. and Joki, K. (2018).
\newblock Bundle-based descent method for nonsmooth multiobjective {DC}
  optimization with inequality constraints.
\newblock {\em Journal of Global Optimization}, 72(3):403--429.

\bibitem[Pang, 2007]{Pang_2007}
Pang, J.~S. (2007).
\newblock Partially {B}-regular optimization and equilibrium problems.
\newblock {\em Mathematics of Operations Research}, 32(3):687--699.

\bibitem[Pang et~al., 2017]{Pang_2017}
Pang, J.-S., Razaviyayn, M., and Alvarado, A. (2017).
\newblock Computing {B}-stationary points of nonsmooth {DC} programs.
\newblock {\em Mathematics of Operations Research}, 42(1):95--118.

\bibitem[Pr\'{e}kopa, 2003]{Prekopa_2003}
Pr\'{e}kopa, A. (2003).
\newblock Probabilistic programming.
\newblock In Ruszczy\'{n}ski, A. and Shapiro, A., editors, {\em Stochastic
  Programming}, volume~10 of {\em Handbooks in Operations Research and
  Management Science}, pages 267--351. Elsevier, Amsterdam.

\bibitem[Rockafellar and Royset, 2010]{Rockafellar_Royset_2010}
Rockafellar, R. and Royset, J. (2010).
\newblock On buffered failure probability in design and optimization of
  structures.
\newblock {\em Reliability Engineering \& System Safety}, 95(5):499--510.

\bibitem[Rockafellar and Wets, 1986]{Rockafellar_Wets_1986}
Rockafellar, R.~T. and Wets, R. J.-B. (1986).
\newblock A {L}agrangian finite generation technique for solving
  linear-quadratic problems in stochastic programming.
\newblock {\em Mathematical Programming Study}, 28:63--93.

\bibitem[Rockafellar and Wets, 2009]{Rockafellar_Wets_2009}
Rockafellar, R.~T. and Wets, R. J.-B. (2009).
\newblock {\em Variational Analysis}, volume 317 of {\em Grundlehren der
  mathematischen Wissenschaften}.
\newblock Springer Verlag Berlin, 3rd edition.

\bibitem[Royset and Wets, 2022]{RoysetWets_2022}
Royset, J. and Wets, R.-B. (2022).
\newblock {\em An optimization primer}.
\newblock Springer Series in Operations Research and Financial Engineering.
  Springer Cham, 1 edition.

\bibitem[Royset, 2023]{Royset_2023}
Royset, J.~O. (2023).
\newblock Consistent approximations in composite optimization.
\newblock {\em Mathematical Programming}, 201(1):339--372.

\bibitem[Sagastiz\'{a}bal, 2012]{Sagastizabal_2012}
Sagastiz\'{a}bal, C. (2012).
\newblock Divide to conquer: Decomposition methods for energy optimization.
\newblock {\em Mathematical Programming}, 134(1):187--222.

\bibitem[Sagastiz{\'a}bal, 2013]{Sagastizabal_2013}
Sagastiz{\'a}bal, C. (2013).
\newblock Composite proximal bundle method.
\newblock {\em Mathematical Programming}, 140(1):189--233.

\bibitem[Sagastiz\'{a}bal and Solodov, 2005]{Sagastizabal_Solodov_2005}
Sagastiz\'{a}bal, C. and Solodov, M. (2005).
\newblock An infeasible bundle method for nonsmooth convex constrained
  optimization without a penalty function or a filter.
\newblock {\em SIAM Journal on Optimization}, 16(1):146--169.

\bibitem[Shapiro et~al., 2009]{SDR09}
Shapiro, A., Dentcheva, D., and Ruszczy\'nski, A. (2009).
\newblock {\em Lectures on Stochastic Programming: Modeling and Theory}.
\newblock MPS-SIAM Series on Optimization. SIAM - Society for Industrial and
  Applied Mathematics and Mathematical Programming Society, Philadelphia.

\bibitem[Song and Kiureghian, 2003]{Song_2003}
Song, J. and Kiureghian, A.~D. (2003).
\newblock Bounds on system reliability by linear programming.
\newblock {\em Journal of Engineering Mechanics}, 129(6):627--636.

\bibitem[Strekalovsky and Minarchenko, 2017]{Strekalovsky_Minarchenko_2017}
Strekalovsky, A. and Minarchenko, I. (2017).
\newblock On local search in {\sc d.c.} optimization.
\newblock In {\em 2017 Constructive Nonsmooth Analysis and Related Topics
  (dedicated to the memory of V.F. Demyanov) (CNSA)}, pages 1--4.

\bibitem[Syrtseva et~al., 2023]{KseniaPaul_2023}
Syrtseva, K., de~Oliveira, W., Demassey, S., Morais, H., Javal, P., and
  Swaminathan, B. (2023).
\newblock Difference-of-convex approach to chance-constrained optimal power
  flow modelling the dso power modulation lever for distribution networks.
\newblock {\em Sustainable Energy, Grids and Networks}, 36.

\bibitem[Syrtseva et~al., 2024]{Ksenia_2023}
Syrtseva, K., de~Oliveira, W., Demassey, S., and van Ackooij, W. (2024).
\newblock Minimizing the difference of convex and weakly convex functions via
  bundle method.
\newblock {\em Pacific Journal of Optimization}.
\newblock DOI: 10.61208/pjo-2024-004.

\bibitem[van Ackooij et~al., 2017]{vanAckooij_Berge_Oliveira_Sagastizabal_2017}
van Ackooij, W., Berge, V., {W. de Oliveira}, and Sagastiz\'abal, C. (2017).
\newblock Probabilistic optimization via approximate p-efficient points and
  bundle methods.
\newblock {\em Computers \& Operations Research}, 77:177 -- 193.

\bibitem[{van Ackooij} and {de Oliveira}, 2019]{Ackooij_Oliveira_OMS_2019}
{van Ackooij}, W. and {de Oliveira}, W. (2019).
\newblock Non-smooth {DC}-constrained optimization: constraint qualification
  and minimizing methodologies.
\newblock {\em Optimization Methods and Software}, 34(4):890--920.

\bibitem[van Ackooij and {de Oliveira}, 2020]{Ackooij_Oliveira_2020}
van Ackooij, W. and {de Oliveira}, W. (2020).
\newblock Some brief observations in minimizing the sum of locally lipschitzian
  functions.
\newblock {\em Optimization Letters}, 14(3):509--520.

\bibitem[{van Ackooij} and {de Oliveira}, 2022]{Ackooij_Oliveira_2022}
{van Ackooij}, W. and {de Oliveira}, W. (2022).
\newblock Addendum to the paper `nonsmooth {DC}-constrained optimization:
  constraint qualification and minimizing methodologies'.
\newblock {\em Optimization Methods and Software}, pages 2241--2250.

\bibitem[{van Ackooij} et~al., 2021]{Javal_2021}
{van Ackooij}, W., Demassey, S., Javal, P., Morais, H., {de Oliveira}, W., and
  Swaminathan, B. (2021).
\newblock A bundle method for nonsmooth {DC} programming with application to
  chance-constrained problems.
\newblock {\em Computational Optimization and Applications}, 78(2):451--490.

\bibitem[van Ackooij et~al., 2014]{vanAckooij_Henrion_Moller_Zorgati_2011b}
van Ackooij, W., Henrion, R., M\"{o}ller, A., and Zorgati, R. (2014).
\newblock Joint chance constrained programming for hydro reservoir management.
\newblock {\em Optimization and Engineering}, 15:509--531.

\end{thebibliography}

\appendix

\section{On the (Inexact) Solution of Subproblems: a Prox-form of Bundle Method}\label{ap:bundle}
In this section, we rely on \cite{Correa_Lemarechal_1993} and show how a bundle-like algorithm can be deployed for computing an $\epsilon^k$-solution $y^k$ of~\eqref{spbm} for the class of models $\Model$ holding Assumption \ref{assump-model}. \wlo{For a deeper analysis on the rates of convergence of this algorithms we rely on \cite{DiazGrimmer_2023}}. Note that, in general, the chosen model need not be convex. Following the idea in Remark \ref{remark:subproblems}, we propose an algorithm to solve every model $\Model_a(\cdot;x^k)$, which happens to be convex.

In what follows, let $k$ be a fixed iteration of Algorithm \ref{algo}, $x^k$ be the current point, and $\Model_a(\cdot;x^k)$ for a fixed $a\in A(x^k)$. In  the following approach, iterations will be denoted by $m$: our goal is to generate a sequence of \emph{inner iterates} $\{y^{k,m}_a\}_m\subset X$ with the following properties:
\begin{enumerate}[label=(\alph*)]
    \item if $x^k$ solves~\eqref{eq:subproblems}, then $\lim_{m\to \infty} y^{k,m}_a=x^k$;
    \item if $x^k$ is not a solution to~\eqref{eq:subproblems}, then $\{y^{k,m}_a\}_m$ is finite and its last term, denoted by $y^k_a$, is an $\epsilon^k_a$-solution of~\eqref{eq:subproblems}, with $\epsilon^k_a$ satisfying 
    $
    \epsilon^k_a\leq\frac{\lambda}{2}\norm{y^k_a-x^k}^2.
    $

\end{enumerate}
To this end, at every inner iteration $m$, we approximate a convex function $\Model_a(\cdot;x^k)$ with a cutting-plane model $\H^{k,m}_a\leq \Model_a(\cdot;x^k)$ and define $y^{k,m+1}_a$  as 
\begin{equation}\label{spbm:app}
y^{k,m+1}_a:=\arg\min_{x\in X}\;\H^{k,m}_a(x)+\frac{\mu^k}{2}\norm{x-x^k}^2.
\end{equation}
As we will precise below, $\H^{k,m}_a$ is given by at most $m$ pieces. Thus, if $X$ is polyhedral, then the above problem is a strictly convex quadratic programming problem, for which efficient (open source and commercial) solvers exist.
As standard in bundle methods (see for instance \cite{Oliveira_Solodov_2020,Frangioni_2020}) the number of cuts in the model $\H^{k,m}_a$ can be bounded. Indeed, only two pieces are enough to have a convergent algorithm. In more general terms, any rule to update $\H^{k,m}_a$ to $\H^{k,m+1}_a$ is a valid one provided that the following inequalities are satisfied:
    \begin{equation}\label{cond:app}
    \begin{matrix}
    \hfill\H^{k,m+1}_a\leq&\Model_a(\cdot;x^k),\\
    \H^{k,m}_a(y^{k,m+1}_a)+\mu^k\inner{x^k-y^{k,m+1}_a}{\cdot-y^{k,m+1}_a}\leq& \H^{k,m+1}_a,\\
    \hfill\Model_a(y^{k,m+1}_a,x^k)+\inner{s^{k,m}_a}{\cdot-y_a^{k,m+1}}\leq&\H^{k,m+1}_a,
    \end{matrix}
 \end{equation}
     where $s^{k,m}_a\in\partial\Model_a(\cdot;x^k)(y^{k,m}_a)$ is an arbitrary subgradient. \wlo{The above three conditions are standard to updating a model. The first one implies that the models underestimate the convex function $\Model_a(\cdot;x^k)$. The two last, seek the updated model $\Model_a^{k,m+1}$ to be tighter than the linearizations of the previous model (note that $\mu^k(x^k-y_a^{k,m+1})\in(\Model_a^{k,m}+{\rm i}_X)(y_a^{k,m+1})$) and the function.}

\begin{theorem}[{\cite[Proposition 4.3]{Correa_Lemarechal_1993}}]\label{Th:C-L}
    Consider a sequence of convex models $\{\H^{k,m}_a\}_m$ such that conditions \eqref{cond:app} hold and sequence $\{y^{k,m}_a\}_m$ given by \eqref{spbm:app}. Then,
       \[
       \lim_{m \to \infty } \Model_a(y^{k,m+1}_a;x^k)-\H_a^{k,m}(y^{k,m+1}_a)=0\quad \mbox{and}\quad \lim_{m \to \infty } y^{k,m}_a=\bar y^k_a, 
       \]
       where $\bar y^k_a$ is the unique solution to \eqref{eq:subproblems}.
\end{theorem}

Given these ingredients, we now present the following algorithm for inexactly solving~\eqref{spbm}.

\begin{algorithm}[htb]
\caption{\sc prox-form of bundle method}
\label{algo app}
\begin{algorithmic}[1]
{\footnotesize
\State  Let $x^k,\,\mu^k,\,\kappa,\,\lambda$, and $\tol$ be provided by Algorithm~\ref{algo} at iteration $k$
\State Set $y^{k,0}_a=x^k$ and define $\H^{k,0}_a:=\Model_a(y_a^{k,0};x^k) + \inner{s_a^{k,0}}{\cdot -y_a^{k,0}}$, with $s_a^{k,0}\in \partial \Model_a(\cdot;x^k)(y_a^{k,0})$
\Statex
\For{$m = 0, 1, 2, \dots$}
\State Let $y^{k,m+1}_a$ be the unique solution of 
    \begin{equation*}
    \min_{x \in X} \; \H_a^{k,m}(x) + \frac{\mu^k}{2}\norm{x-x^k}^2
    \end{equation*}
\If{$\Model_a(x^k;x^k)-\H_a^{k,m}(y_a^{k,m+1})\leq \tol$}
\State Stop and return $y^k_a:=x^k$\label{lin:xk}
\EndIf    
\If{$\Model_a(y_a^{k,m+1};x^k)-\H_a^{k,m}(y_a^{k,m+1})\leq \frac{\lambda}{2}\norm{y_a^{k,m+1}-x^k}^2$}\label{sstep}
\State Stop and return $y_a^k:=y_a^{k,m+1}$\label{lin:ykm}
\EndIf
\Statex
\State Update the model $\H_a^{k,m}$ to $\H_a^{k,m+1}$ in order to satisfy \eqref{cond:app}
\EndFor
}
  \end{algorithmic}
\end{algorithm}

Let  $s_a^{k,j} \in \partial \Model_a(\cdot;x^k)(y_a^{k,j})$, $j=1,\ldots,m$.
As standard in bundle methods, a practical rule to update the model is as follows:
\[
\H_a^{k,m+1}(\cdot):= \max_{j \in J_{m+1}}\{\Model_a(y^{k,j};x^k)+ \inner{s_a^{k,j}}{\cdot - y_a^{k,j}}\},
\]
where $J_{m+1}:=\{{J^{act}_m,m+1}\}$, and $J^{act}_m \subset J_m\subset \{0,1,\ldots,m\}$ is the index set of active constraints of the following reformulation of problem~\eqref{spbm}:
\[
\left\{
\begin{array}{lll}
\displaystyle\min_{y\in X,r\in \R}& r + \frac{\mu^k}{2}\norm{x-x^k}^2\\
\mbox{s.t.} & \Model_a(y_a^{k,j};x^k)+ \inner{s_a^{k,j}}{x - y_a^{k,j}}\leq r& \forall\, j \in J_m.
\end{array}
\right.
\]
It is not difficult to show that this rule satisfies~\eqref{cond:app}. \wlo{Check \cite[Remark 4.2]{Correa_Lemarechal_1993} to have a better insight on this type of models.}

\begin{theorem}\label{th:subconv}
    Consider Algorithm~\ref{algo app} with $\tol=0$ applied to problem~\eqref{eq:subproblems}.
    \begin{enumerate}[itemsep=2pt, topsep=2pt, leftmargin=1.75cm,label=(\roman*)]
       \item If the algorithm loops indefinitely, or it stops at line \ref{lin:xk}, then  $x^k$ is a solution to~\eqref{eq:subproblems}.
       
        \item If the algorithm halts at line \ref{lin:ykm}, then $y^k_a:=y_a^{k,m+1}$ is an $\epsilon^k_a$-solution of~\eqref{eq:subproblems}, with $\epsilon_a^k:= \Model_a(y_a^{k,m+1};x^k)-\H_a^{k,m}(y_a^{k,m+1})\geq 0$ satisfying $\epsilon^k_a\leq\frac{\lambda}{2}\norm{y^k_a-x^k}^2$.
        \item  Suppose that $x^k$ does not solve~\eqref{eq:subproblems}.
    Then Algorithm \ref{algo app} stops after finitely many iterations.
    \end{enumerate}
\end{theorem}
{\bf Proof.}
To prove item (i), note that if the algorithm stops at line~\ref{lin:xk}, then
\begin{align*}
 \H_a^{k,m}(y_a^{k,m+1})+  {\frac{\mu^k}{2}}\norm{y_a^{k,m+1}-x^k}^2  & =
     \min_{y \in X} \left\{ \H_a^{k,m}(y) + \frac{\mu^k}{2}\norm{y-x^k}^2\right\}\\
     &\leq  \min_{y \in X} \left\{ \Model_a(y;x^k) + \frac{\mu^k}{2}\norm{y-x^k}^2\right\}\\
     & \leq  \Model_a(x^k;x^k)\leq  \H_a^{k,m}(y_a^{k,m+1}) + \tol.
\end{align*}
In this case, $\tol =0$ implies that $y_a^{k,m+1}=x^k$ solves~\eqref{eq:subproblems}.
On the other hand, if the algorithm loops indefinitely, then
\[ \frac{\lambda}{2}\norm{y^{k,m+1}_a-x^k}^2 <    \Model_a(y_a^{k,m+1};x^k)-\H_a^{k,m}(y_a^{k,m+1})
\quad \forall\, m.
\]
Theorem~\ref{Th:C-L} then ensures that $\lim_{m \to \infty} y^{k,m}_a=x^k$ solves~\eqref{eq:subproblems}.

To show item (ii), note that by setting $\epsilon_a^k:=\Model_a(y_a^{k,m+1};x^k) -\H_a^{k,m}(y^{k,m+1}) \geq 0$ we get
\begin{align*}
\Model_a(y^{k,m+1}_a;x^k) + \frac{\mu^k}{2}\norm{y_a^{k,m+1}-x^k}^2 &=  \H^{k,m}_a(y_a^{k,m+1})+ {\frac{\mu^k}{2}}\norm{y_a^{k,m+1}-x^k}^2  + \epsilon^k_a\\ 
& =
     \min_{x \in X} \left\{ \H_a^{k,m}(x) + \frac{\mu^k}{2}\norm{x-x^k}^2\right\}+ \epsilon^k_a\\
     &\leq  \min_{x \in X} \left\{ \Model_a(x;x^k) + \frac{\mu^k}{2}\norm{x-x^k}^2\right\}+ \epsilon_a^k,
\end{align*}
i.e., $y^{k,m+1}_a$ is an $\epsilon_a^k$-solution to~\eqref{eq:subproblems}. By assumption the algorithm stops at line~\ref{lin:ykm}, then $\epsilon^k_a\leq \frac{\lambda}{2}\norm{y^{k,m+1}_a-x^k}^2$
and the result follows.

Finally, to prove item (iii), see that if $x^k$ does not solve~\eqref{eq:subproblems}, then Theorem \wlo{\ref{th:subconv}} ensures that  for every $\epsilon>0$ small enough, there exists $m_0\in \mathbb N$ such that
$$
\Model_a(y^{k,m+1}_a;x^k)-\H^{k,m}_a(y_a^{k,m+1})<\epsilon
\quad \mbox{and} \quad 
\epsilon<\frac{\lambda}{2}\norm{y^{k,m+1}_a-x^k}^2
\quad\mbox{for all }\; m\geq m_0,
$$
where $\lambda$ is the one in Algorithm \ref{algo app}. 
Hence, the stopping test at line~\ref{sstep} is triggered and the algorithm stops after finitely many steps.
\mybox

\begin{remark}
    We need to apply Algorithm \ref{algo app} for all $a\in A(x^k)$ in order to compute an $\epsilon^k_a$-solution to every problem \eqref{eq:subproblems}. Once all these iterations are done, we compute $y^k$ as
    $$
    y^k\in\arg\min_{y\in Y^k}\Model(y;x^k)+\frac{\mu^k}{2}\norm{y-x^k}^2,
    $$
    where $Y^k=\{y^k_a:a\in A(x^k)\}$ is a finite set. We can easily prove that $y^k$ is such that
    $$
    \Model(y^k;x^k)+\frac{\mu^k}{2}\norm{y^k-x^k}^2\leq \Model(y^k_a;x^k)+\frac{\mu^k}{2}\norm{y^k_a-x^k}^2\leq  \Model_a(y^k_a;x^k)+\frac{\mu^k}{2}\norm{y^k_a-x^k}^2\leq \Model_a(x;x^k)+\frac{\mu^k}{2}\norm{x-x^k}^2+\epsilon^k_a
    $$
    for all $a\in A(x^k)$ and $x\in X$. Then, minimizing over $x\in X$ and $a\in A(x^k)$ the right-hand side of the above chain of inequalities, we deduce
    $$
    \Model(y^k;x^k)+\frac{\mu^k}{2}\norm{y^k-x^k}^2\leq \min_{x\in X}\Model(x;x^k)+\frac{\mu^k}{2}\norm{x-x^k}^2+\epsilon^k,
    $$
     where $\epsilon^k=\min_{a\in A(x^k)}\epsilon^k_a$. As $y^k=y^k_{a^*}$ for a given $a^*\in A(x^k)$, we conclude that $y^k$ is an $\epsilon^k$-solution to problem \eqref{spbm} with 
     $$\epsilon^k\leq \epsilon^k_{a^*}\leq\frac{\lambda}{2}\norm{y^k-x^k}^2.$$\mytriangle
\end{remark}
\end{document}